\documentclass[12pt]{amsart}
\usepackage{amsmath}
\usepackage{amssymb}
\usepackage{epsfig}
\usepackage{wasysym}
\usepackage{textcomp}
\usepackage{graphicx}
\usepackage{tikz}
\usetikzlibrary{calc}
\usetikzlibrary{intersections}
\usepackage{tikz-cd}
\usepackage{bm}
\usepackage{xcolor}
\usepackage{listings}
\numberwithin{equation}{section}
\usepackage{extpfeil}
\usepackage{array}
\usepackage{adjustbox}
\usepackage{longtable}
\usepackage{makecell}
\usepackage[all]{xy}

\input xy
\xyoption{all}

\usetikzlibrary{positioning}
\usepackage[colorlinks=false,urlbordercolor=white]{hyperref}
  
\tikzset{sgplattice/.style={inner sep=1pt,norm/.style={red!50!blue},char/.style={blue!50!black},
  lin/.style={black!50}},cnj/.style={black!50,yshift=-2.5pt,left=-1pt of #1,scale=0.5,fill=white}}

\DeclareFontFamily{U}{mathb}{\hyphenchar\font45}
\DeclareFontShape{U}{mathb}{m}{n}{
      <5> <6> <7> <8> <9> <10> gen * mathb
      <10.95> mathb10 <12> <14.4> <17.28> <20.74> <24.88> mathb12
      }{}
\DeclareSymbolFont{mathb}{U}{mathb}{m}{n}
\DeclareMathSymbol{\righttoleftarrow}{3}{mathb}{"FD}

\calclayout
\allowdisplaybreaks[3]

\theoremstyle{plain}
\newtheorem{prop}{Proposition}[section]

\newtheorem{theon}{Theorem}
\newtheorem{theo}[prop]{Theorem}
\newtheorem{coro}[prop]{Corollary}
\newtheorem{corol}[theon]{Corollary}

\newtheorem{lemm}[prop]{Lemma}

\theoremstyle{definition}

\newtheorem{conj}[prop]{Conjecture}

\newtheorem{rema}[prop]{Remark}

\newtheorem{exam}[prop]{Example}

\newcommand{\actsfromleft}{\mathrel{\reflectbox{$\righttoleftarrow$}}}
\newcommand{\actsfromright}{\righttoleftarrow}

\def\lra{\longrightarrow}

\def\cM{{\mathcal M}}

\def\fA{{\mathfrak A}}

\def\fD{{\mathfrak D}}

\def\fS{{\mathfrak S}}

\def\fS{{\mathfrak S}}

\def\bG{{\mathbb G}}
\def\bP{{\mathbb P}}
\def\bQ{{\mathbb Q}}

\def\bZ{{\mathbb Z}}

\def\rH{{\mathrm H}}

\def\bF{{\mathbb F}}

\def\Pic{\mathrm{Pic}}

\def\Aut{\mathrm{Aut}}

\def\SL{\mathsf{SL}}
\def\GL{\mathsf{GL}}

\def\PGL{\mathsf{PGL}}

\def\lim{\mathrm{lim}}

\def\Cr{\mathrm{Cr}}

\makeatother
\makeatletter

\begin{document}

\title[Equivariant geometry of threefolds]{Equivariant geometry of the Segre cubic and the Burkhardt quartic}

\author{Ivan Cheltsov}
\address{Department of Mathematics, University of Edinburgh, UK}

\email{I.Cheltsov@ed.ac.uk}

\author{Yuri Tschinkel}
\address{
  Courant Institute,
  251 Mercer Street,
  New York, NY 10012, USA}
  
\address{Simons Foundation,
160 Fifth Avenue,
New York, NY 10010,
USA}
\email{tschinkel@cims.nyu.edu}

\author{Zhijia Zhang}
\address{
Courant Institute,
  251 Mercer Street,
  New York, NY 10012, USA
}

\email{zhijia.zhang@cims.nyu.edu}

\date{\today}

\begin{abstract}
We study linearizability and stable linearizability of actions of finite groups on the Segre cubic and Burkhardt quartic, using techniques from group cohomology, birational rigidity, and the Burnside formalism. 
\end{abstract}

\maketitle

\section{Introduction}
\label{sect:intro}

Let $G$ be a finite group. 
We study generically free $G$-actions on rational Fano threefolds, 
over an algebraically closed field of characteristic zero, up to equivariant birationality. This is part of a long-standing program to identify finite subgroups of the Cremona group $\mathrm{Cr}_3$ 
(see, e.g., \cite{Pro-ICM} for background and references concerning this problem). One of the main tools in this area of research is the Equivariant Minimal Model Program (EMMP), and in particular the study of {\em birational rigidity} (BR). Among the principal achievements is the classification of finite simple groups that can act on rationally connected threefolds \cite{ProSimple}. There is a wealth of results towards distinguishing conjugacy classes of embeddings of simple groups into the Cremona group, e.g., $\fA_5$ (see \cite{CS}).
There are also many interesting problems: even the classification of involutions in $\Cr_3$ is still open \cite{Pro-inv}. 

There are two particularly intriguing examples of rational threefolds with large automorphisms: the 
 {\em Segre cubic} $X_3\subset \bP^4\subset \bP^5$, given by 
\begin{equation}
\label{eqn:segre}
\sum_{i=1}^6 x_i^3 = \sum_{i=1}^6 x_i =0, 
\end{equation}
with $\Aut(X_3)=\fS_6$, acting via permutation of variables, and   
the {\em Burkhardt quartic} $X_4$ that can be defined in $\bP^4\subset \bP^5$ by the vanishing of elementary symmetric polynomials in $6$ variables of degree $1$ and $4$
\begin{equation}
\label{eqn:quart}
\sum_{1\le i<j<k<l\le 6} x_ix_jx_kx_l = \sum_{i=1}^6 x_i =0, 
\end{equation}
and as such carries the action of $\fS_6$. 
However, the full automorphism group of the Burkhardt quartic is $\mathsf{PSp}_4(\bF_3)$, of order 25920.
Another standard form of the Burkhardt quartic is 
\begin{equation}
\label{eqn:burg}
\big\{y_1(y_1^3+y_2^3+y_3^3+y_4^3+y_5^3)+3y_2y_3y_4y_5=0\big\}\subset\mathbb{P}^4,
\end{equation}
which we will often use in this paper.

Our goal is to identify subgroups in $\fS_6$ and $\mathsf{PSp}_4(\bF_3)$ 
whose actions on $X_3$ and $X_4$ are (projectively) {\em nonlinearizable}, i.e., not equivariantly birational to linear (or projectively linear) actions on~$\mathbb{P}^3$.
To do this, we explore the range of applicability of group cohomology, birational rigidity, and the Burnside group formalism \cite{BnG}.

\

Our main results are:

\begin{theon}[Theorem~\ref{thm:segrelin}]\label{theo:mainsegre}
Let $G\subseteq \fS_6=\Aut(X_3)$. 
The $G$-action on $X_3$ is linearizable if and only if
one of the following conditions holds:
\begin{itemize}
\item $G$ fixes a singular point on $X_3$, 
\item $G$ is contained in the nonstandard $\fS_5'$, up to conjugation,  
\item $G=C_2^2$, $X_3$ contains three $G$-invariant planes, and $\mathrm{Sing}(X)$ splits as a union of five $C_2^2$-orbits of length $2$.
\end{itemize}
Moreover, when the $G$-action is not linearizable, it is not stably linearizable.
\end{theon}

There are $55$ conjugacy classes of nontrivial subgroups $G\subseteq \fS_6$, and $19$ of these give rise to
nonlinearizable actions on $X_3$.

\begin{theon}\label{theo:mainburk}
Let $G\subseteq \mathsf{PSp}_4(\bF_3)=\Aut(X_4)$. 
The $G$-action on $X_4$ is nonlinearizable if at least one of the following conditions holds:
\begin{enumerate}
\item $\mathrm{rk}\mathrm{Cl}(X_4)^G=1$,
\item $G$ contains an involution that swaps two coordinates in $\mathbb{P}^5$,  
\item $G$ contains a subgroup $G'$ such that $\rH^1(G', \Pic(\widetilde{X}_4))\ne 0$, 
\end{enumerate}
where $\widetilde{X}_4$ is the standard resolution of $X_4$.
\end{theon}

This gives nonlinearizability for $103$ out of $115$ conjugacy classes of nontrivial subgroups of $\mathsf{PSp}_4(\bF_3)$.

\begin{theon}
\label{theo:mainburkk}
Among the remaining $12$ conjugacy classes of subgroups 
$G\subset \mathsf{PSp}_4(\bF_3)$, the $G$-action is linearizable if $G$ 
is conjugate to one of the subgroups 
$$
C_2, C_3, C_3', C_2^2, C_4, C_5, C_6, C_9,
$$
explicitly described in Section~\ref{sect:burk}.
\end{theon}

See \cite{CTZ-tables} for additional information concerning the actions of the groups in Theorems~\ref{theo:mainsegre}, \ref{theo:mainburk}, \ref{theo:mainburkk}.

Thus, we settle completely the (stable) linearizability problem for the Segre cubic threefold $X_3$. After excluding known cases when either there is a $G$-fixed singular point in $X_3$, and the action is linearizable via projection from this point, or $G$ is conjugated to a subgroup of the nonstandard $\fS_5$ in $\fS_6$, and such actions have been treated in \cite{Avilov,HT-torsor}, the remaining analysis hinges on the existence of $G$-stable planes. 
A  key observation in the remaining cases
is that if $X_3$ does {\em not} contain a $G$-stable plane, 
then there is a cohomological obstruction to stable linearizability, and if it does contain such a plane, then $X_3$ is birational to a singular, toric, intersection of two quadrics in $\bP^5$, which can be analyzed via toric geometry. 

As an auxiliary tool, we settle the (stable) linearizability problem for translation-free actions with fixed points 
on algebraic tori in dimension $3$, in Theorem~\ref{thm:lin-tori}. 

For the Burkhardt quartic $X_4$, we settle the linearizability problem for all finite subgroups of $\mathsf{PSp}_4(\bF_3)$
except for $4$ (conjugacy classes of) subgroups isomorphic to 
$$
\mathfrak{S}_3,\,\fD_5, \,\fD_6, \, C_3\rtimes C_4,
$$
which are described explicitly in Section~\ref{sect:burk}. 
We do not know whether or not the actions of these $4$ groups are linearizable. 
For the subgroups $\mathfrak{S}_3$ and $\fD_6$ we present an equivariant birational map from $X_4$ to a smooth quadric threefold,
see Section~\ref{sect:burk}.

The proofs of Theorems~\ref{theo:mainsegre}, \ref{theo:mainburk}, and \ref{theo:mainburkk} imply the following corollary.

\begin{corol}
\label{corol:main}
Let $G$ be a group acting faithfully on $X_3$ and $X_4$ and such that 
both actions are not linearizable. 
Then there is no $G$-equivariant birational map $X_3\dasharrow X_4$, 
with the possible exception when $G\simeq C_2^2$, 
conjugate to 
$$
\big\langle(1\ 2)(3\ 4),(1\ 2)(5\ 6)\big\rangle\subset \fS_6=\mathrm{Aut}(X_3),
$$
and the corresponding subgroup in $\mathrm{Aut}(X_4)\subset \PGL_5$ generated by 
$$
\begin{pmatrix}
-1 & 2 & 2 & 2& 2\\
1 & 1 & -2 & 1 & 1\\
1 & -2 & 1 & 1 & 1\\
1 & 1 & 1 & 1 & -2\\
1 & 1 & 1 & -2 & 1
\end{pmatrix},
\begin{pmatrix}
1 & 0 & 0 & 0 & 0\\
0 & 0 & 1 & 0 & 0\\
0 & 1& 0& 0& 0\\
0 & 0 & 0 &0&1\\
0 &0 &0 &1 &0
\end{pmatrix}.
$$    
\end{corol}

It would be interesting to clarify what happens in this exceptional case. 

\

Here is the roadmap of the paper: in Section~\ref{sect:gen} we recall basic notions and constructions in equivariant birational geometry, and tabulate results of computations for the Segre cubic and the Burkhardt quartic. In Section~\ref{sect:bir} we recall the main  tools from equivariant birational rigidity. Section~\ref{sect:incomp} presents a simplified version of the Burnside formalism from \cite{BnG}, based on incompressible symbols, that allows to prove new cases of nonlinearizability.
In Section~\ref{sect:tori} we adopt Kunyavski's  
rationality analysis of 3-dimensional tori over nonclosed fields \cite{kun3} to the equivariant context. 
In Section~\ref{sect:segre} we turn to the Segre cubic threefold; we address the linearizability of the Burkhardt quartic in Section~\ref{sect:burk}. 
In Section~\ref{sect:Burk-rigid}, we prove that the Burkhardt quartic $X_4$ is $G$-birationally rigid if 
$G\subseteq \mathsf{PSp}_4(\bF_3)$ satisfies $\mathrm{rk}\,\mathrm{Cl}(X_4)^G=1$. 

\

Throughout, we consider group actions from the right, by ${\tt Magma}$ conventions; $C_n$ denotes cyclic groups of order $n$; $\fD_n$ denotes dihedral groups of order $2n$; $\fA_n$ and $\fS_n$ denote alternating groups and symmetric groups of degree $n$, respectively. 

\

\noindent
{\bf Acknowledgments:} 
The first author was partially supported by the Leverhulme Trust grant RPG-2021-229.
The second author was partially supported by NSF grant 2000099.

\section{Background}
\label{sect:gen}

\subsection*{Linearizability}

Let $G$ be a group and $X$ a $G$-variety, i.e., an algebraic variety with a generically free action of $G$. 
We are interested in the following properties of $G$-actions: 
\begin{itemize}
\item {\em Linearizability:} $X$ is $G$-equivariantly 
birational to $\bP(V)$, the projectivization of a linear representation $V$ of $G$;
\item {\em Stable linearizability:} $X\times \bP^m$, with trivial $G$-action on the second factor, is linearizable.
\end{itemize}
One may also consider the related notions of (stable) {\em projective} linearizability, where the $G$-action on $X$ is compared to the $G$-action on 
$\bP(V)$, the projectivization of a representation $V$ of a {\em central extension} of $G$. Our focus in this paper is on linearizability, since
projectively linear actions on the Segre cubic and the Burkhardt quartic are linear.

Equivariant resolution of singularities (over fields of characteristic zero) allows to reduce to the case when $X$ is smooth.

\subsection*{Obstructions}
Let $X$ be a smooth projective $G$-variety, over an algebraically closed field of characteristic zero. The geometric action induces an action on invariants of $X$, such as the Picard group $\Pic(X)$.
Among necessary conditions for stable linearizability of the $G$-action on $X$ is
\begin{itemize}
\item[{\bf(SP)}] $\Pic(X)$ is a stably permutation $G$-module.
\end{itemize}
This condition is not easy to verify, in practice. On the other hand, it implies the more tractable condition
\begin{itemize}
\item[{\bf (H1)}] 
$\rH^1(G', \Pic(X)) = \rH^1(G', \Pic(X)^\vee)=0, \quad \forall G'\subseteq G.$
\end{itemize}
This can be checked in {\tt Magma}, when the $G$-action on $\Pic(X)$ is known explicitly. 
Note that these conditions are equivalent for $G=C_2$. 

It should be pointed out that not all failures of birationality are explained by invariants of the $G$-action on Picard groups, see Section~\ref{sect:incomp} and Remark~\ref{rem:bbur}. 
For singular $G$-varieties it is also useful to consider the induced $G$-action on the class group $\mathrm{Cl}(X)$; this is particularly relevant to the study of $G$-birational rigidity, see Section~ \ref{sect:bir} and \ref{sect:Burk-rigid}. 

\begin{rema}
If $X$ is a singular $G$-variety, then $\rH^1(G, \mathrm{Cl}(X))$ is not a $G$-birational invariant. For instance, let $X=X_3$ be the Segre cubic in $\bP^5$ and $G=C_2$, acting on $X$ via swapping of two coordinates. Then $\mathrm{Cl}(X)=\bZ^6$ and $\rH^1(G,\mathrm{Cl}(X))=\bZ/2$. But $G$ fixes a singular point on $X$ and is thus linearizable.
\end{rema}

\subsection*{Rational surfaces}
Actions of finite groups on Del Pezzo surfaces have been extensively studied in \cite{DI,prokhorovII}. 
The $G$-action on Picard groups of Del Pezzo surfaces of degrees $4,3,2,1$ factors through subgroups of 
Weyl groups 
$$
W(\mathsf D_5), \quad W(\mathsf E_6), \quad W(\mathsf E_7),\quad  W(\mathsf E_8),
$$
respectively. Subgroups satisfying {\bf (H1)} have been enumerated in \cite{TY};  
the paper \cite{prokhorovII} contains examples of such subgroups of $W(\mathsf D_n)$, acting on Picard groups of conic bundles over $\bP^1$. 

Applications of the Burnside formalism from \cite{BnG} to threefolds require a detailed understanding of birationality of $G$-actions on surfaces. 

\subsection*{Segre cubic}

Let $X_3$ be the Segre cubic in $\mathbb{P}^4$.
We have the following $\fS_6$-equivariant diagram
$$
\xymatrix{
&\widetilde{X}_3\ar[ld]_{f}\ar[rd]^{g}&\\
X_3&&X_4^I}
$$
where $f$ is the blowup of the $10$ singular points of the cubic, $g$ is the anticanonical morphism, and $X_4^I$ is the Igusa quartic threefold in $\mathbb{P}^4$.
We say that $f$ is the standard resolution of singularities of $X_3$.
Recall that $\widetilde{X}_3$ is $\fS_6$-equivariantly isomorphic to $\bar{\mathcal M}_{0,6}$, the moduli space of 6 points on $\bP^1$, 
which has a natural $\fS_6$-action permuting the 6 points. 

The group $\fS_6$ has an outer automorphism, and thus two conjugacy classes of 
subgroups $\fS_5$; one of them we call {\em standard}, it acts trivially on one of the indices, and the other {\em nonstandard}; we shall denote it by~$\fS_5'$. By \cite[Proposition 4.1]{Avilov}, the action of the standard $\fS_5$ on the Segre cubic, via permutation of $5$ variables in \eqref{eqn:segre}, is birationally rigid; the action of the nonstandard~$\fS_5'$ is linearizable. 

Recall that $\Pic(\widetilde{X}_3)^{\fS_6}$ is generated by 2 classes, corresponding to the birational contractions to the Segre cubic and the Igusa quartic. The following table provides additional information about ranks of the invariant Picard group and class group, as one changes the action: 

\

{\small

\begin{center}
\renewcommand{\arraystretch}{1.3}
\begin{tabular}{c|c|c|cc|cc|c|c|c}
{\rm Group}                &  $\fS_6$ & $\fA_6$ &   $\fS_5$ &  $\fS_5'$ & $\fA_5$ & $\fA_5'$ & $\fS_3\wr C_2$ &  $C_2\times \fS_4$ & $C_2\times \fS_4''$ \\
\hline
$\mathrm{rk}\,\Pic(\widetilde{X}_3)^G$ &   2      &   2    &  2      &    3      &  2   &  3 & 3 & 3 & 4\\
\hline
$\mathrm{rk}\,\mathrm{Cl}(X_3)^G$ & 1&1&1&2&1&2&1&1&2
\end{tabular}
\end{center}
}
\

\subsection*{Burkhardt quartic}

We also record ranks of invariants in the Picard group of $\widetilde{X}_4$, the standard resolution of singularities of the Burkhardt quartic $X_4$ obtained by blowing up all its singular points, for various subgroups $G\subseteq \mathsf{PSp}_4(\bF_3)$: 

\

{\small

\begin{center}
\renewcommand{\arraystretch}{1.3}
\begin{tabular}{c|c|c|c|c|c|c}
{\rm Group}                        & $\mathsf{PSp}_4(\bF_3)$ & $C_2^4.\fA_5$   & $\fS_6$ & $C_2.\fA_4\wr C_2$  &  $\mathrm{SL}_2(\bF_3):\fA_4$ &   $\mathrm{SL}_2(\bF_3)$\\
\hline
$\mathrm{rk}\,\Pic(\widetilde{X}_4)^G$ &     2                    &    3           &    3   &   4      & 5 &  7,9,11\\
\hline
$\mathrm{rk}\,\mathrm{Cl}(X_4)^G$ & 1&1&1&1&2&3,3,5
\end{tabular}
\end{center}
}

\

\noindent
where the last entry reflects the different conjugacy classes. The full table, obtained with {\tt Magma}, is available at \cite{CTZ-tables}. Furthermore, we have:

\begin{prop}
\label{prop:inv-class}
Let $X_4$ be the Burkhardt quartic and $G\subseteq \mathsf{PSp}_4(\bF_3)$
a subgroup of its automorphism group. Then  $\mathrm{rk\, Cl}(X_4)^G=1$ 
if and only if $G$ contains a subgroup conjugate to 
one of the following subgroups
\begin{align*}
    \mathfrak F_5, \,\, C_2^4,\,\, C_2^2\rtimes C_4, \,\,C_2\times\fD_4,\,\,\fS_4,\,\, \fS_4',\,\,C_3\rtimes\fD_4,\,\,\fS_3^2,\,\,C_3^2\rtimes C_4,\,\,\fA_5,
\end{align*}
explicitly specified in  \cite{CTZ-tables}. 
\end{prop}

This extends \cite[Corollary 2.10]{Ch-Burk} and \cite[Corollary 5.4]{CKS}, which listed the corresponding subgroups of $\fS_6$.

\section{Birational rigidity}
\label{sect:bir}

Let $X$ be a Fano threefold with at most terminal singularities and $G\subseteq \mathrm{Aut}(X)$ a finite subgroup. Suppose that $$
\mathrm{rk}\,\mathrm{Pic}^G(X)=1.
$$
If $X$ is smooth, then $X$ is a $G$-Mori fiber space (over a point), and $X$ lies in $25$ deformation families described in \cite[Theorem~1.2]{Prokhorov2013}.
If $X$ is singular, then it may fail to be $G\mathbb{Q}$-factorial, i.e., we may have
$$
\mathrm{rk}\,\mathrm{Cl}^G(X)>\mathrm{rk}\,\mathrm{Pic}^G(X),
$$
so that $X$ is not necessarily a $G$-Mori fiber space. However, we can always take a $G\mathbb{Q}$-factorialization of $X$,
and then apply EMMP to obtain a $G$-equivariant birational map from $X$ to some $G$-Mori fiber space.

On the other hand, if $\mathrm{rk}\,\mathrm{Cl}^G(X)=1$, then  $X$ is a $G$-Mori fiber space.
In this case, one can try to describe all $G$-birational maps from $X$ to other $G$-Mori fiber spaces.
Every such map can be decomposed into a sequence of $G$-Sarkisov links \cite{Corti1995,HaconMcKernan2013},
which have a more restricted structure.
If there are no $G$-Sarkisov links that start at $X$,
then $X$ is the only  $G$-Mori fiber space that is $G$-birational to $X$ and
$$
\mathrm{Bir}^G(X)=\mathrm{Aut}^G(X),
$$
i.e., $X$ is $G$-birationally super-rigid.
We say that $X$ is $G$-birationally rigid if every $G$-Sarkisov link that starts at $X$ also ends at $X$,
which means that $X$ is not $G$-birational to other $G$-Mori fiber spaces, but $X$ may admit non-biregular $G$-birational selfmaps.

\begin{rema}
\label{rema:rigid-lin}
If $X\not\simeq\mathbb{P}^3$, $\mathrm{rk}\,\mathrm{Cl}^G(X)=1$, and $X$ is $G$-birationally rigid,
then the $G$-action on $X$ is not (projectively) linearizable.
\end{rema}

If $X$ is not $G$-birational to any $G$-Mori fiber space with a positive dimensional base (a conic bundle or a Del Pezzo fibration),
we say that $X$ is $G$-solid. $G$-birationally rigid and $G$-solid Fano threefolds are studied in \cite{CheltsovPark2010,Ch-Burk,ChS-5,CheShr,CS,CS-finite,Ch-toric,CSar,Cheltsov2023},
with a special focus on rational threefolds. These studies are based on the following technical result,
which is the engine of the~$G$-equivariant Sarkisov program:

\begin{theo}[{\cite[Theorem 3.3.1]{CS}}]
\label{theo:NFI}
Suppose that $\mathrm{rk}\,\mathrm{Cl}^G(X)=1$, and
let $\chi\colon X\dasharrow V$ be a $G$-birational non-biregular map such that 
\begin{itemize}
\item 
$V$ has terminal singularities,
\item 
$\mathrm{rk}\,\mathrm{Cl}^G(V)=\mathrm{rk}\,\mathrm{Pic}^G(V)$, and 
\item there exists a $G$-equivariant Mori fiber space $\pi\colon V\to Z$.
\end{itemize}
Set
$$
\mathcal{M}=\chi_*^{-1}\big(|-pK_V+\pi^*(H)|\big),
$$
for $p\gg 0$, and  a sufficiently general very ample divisor $H\in\mathrm{Pic}(Z)$ such that $[\pi^*(H)]$ is $G$-invariant.
Then $\mathcal{M}$ is a $G$-invariant non-empty mobile linear system,
and the singularities of the log pair $(X,\lambda\mathcal{M})$ are not canonical
for $\lambda\in\mathbb{Q}_{>0}$ such that $\lambda\mathcal{M}\sim_{\mathbb{Q}}-K_X$.
\end{theo}

This is a $G$-equivariant version of the classical Noether--Fano inequality.
EMMP and Theorem~\ref{theo:NFI} give a simple criterion for $G$-birational super-rigidity:

\begin{coro}
\label{coro:NFI}
Suppose that $\mathrm{rk}\,\mathrm{Cl}^G(X)=1$.
Then $X$ is $G$-birationally super-rigid if and only if for every $G$-invariant non-empty mobile linear system $\mathcal{M}$ on $X$,
the log pair $(X,\lambda\mathcal{M})$ has canonical singularities for $\lambda\in\mathbb{Q}_{>0}$ such that $\lambda\mathcal{M}\sim_{\mathbb{Q}} -K_X$.
\end{coro}

There is a similar (albeit more technical) criterion for $G$-birational rigidity, see \cite[Chapter 3]{CS}.
If $X$ is toric, and $G$ contains the maximal torus in $\mathrm{Aut}(X)$, a criterion for $G$-solidity is given in \cite{Ch-toric}.

Usually, Corollary~\ref{coro:NFI} is applicable when $(-K_X)^3$ is ``sufficiently small''
or when the group $G$ is ``sufficiently large''.
For instance, for $(-K_X)^3=2$, 
arguing as in the proof of \cite[Theorem~A]{CheltsovPark2010}, we obtain:

\begin{theo}
\label{theo:CP}
Let $X\subset \mathbb{P}(1,1,1,1,3)$ be a hypersurface of degree $6$ with at most isolated ordinary double points (nodes)
and $G\subseteq \mathrm{Aut}(X)$ a finite subgroup
such that 
$\mathrm{rk}\,\mathrm{Cl}(X)^G=1$.
Then $X$ is $G$-birationally super-rigid.
\end{theo}

This is also expected for nodal quartics,
where $(-K_X)^3=4$, see \cite{Mella}:

\begin{conj}[{\cite[Conjecture 5.2]{CKS}}]
\label{conj:CKS}
Let $X\subset \bP^4$ be a nodal quartic threefold
and $G\subseteq \mathrm{Aut}(X)$ a finite subgroup
such that \mbox{$\mathrm{rk}\,\mathrm{Cl}(X)^G=1$}.
Then $X$ is $G$-birationally rigid.
\end{conj}

In Section~\ref{sect:burk}, we prove this conjecture for the Burkhardt quartic.
Unfortunately, we do not have such precise (conjectural) characterizations of $G$-birational rigidity for most 
of the other (singular) Fano threefolds, apart from
sporadic results in this direction.
For instance, for the Segre cubic $X_3$, 
where $(-K_{X_3})^3=24$,
Avilov found all possibilities for $G\subset\mathrm{Aut}(X_3)\simeq\fS_6$ such that $X_3$ is $G$-birationally rigid:

\begin{theo}[{\cite{avilov-forms}}]
\label{theo:Avilov}
Let $X_3\subset \mathbb{P}^4$ be the Segre cubic and
$G\subseteq \mathrm{Aut}(X_3)$ a subgroup such that 
$\mathrm{rk}\,\mathrm{Cl}(X_3)^G=1$.
Then the following are equivalent:
\begin{enumerate}
\item $X_3$ is $G$-birationally rigid,
\item $X_3$ is $G$-birationally super-rigid,
\item $G$ contains a group isomorphic to $\fA_5$ that leaves invariant a hyperplane section of $X_3$.
\end{enumerate}
\end{theo}

Returning to general threefolds, if $\mathrm{rk}\,\mathrm{Cl}(X)^G=2$,
then $X$ admits exactly two $G\mathbb{Q}$-factorializations,
and we have the~following $G$-equivariant commutative diagram:
\begin{equation}
\label{eq:G-Sarkisov}
\xymatrix{
&&V\ar@{-->}[rr]^{\varsigma}\ar@{->}[dll]_{\varphi}\ar@{->}[dr]_{\varpi}&&V^\prime\ar@{->}[dl]^{\varpi^\prime}\ar@{->}[drr]^{\varphi^\prime}\\%
Z&&&X&&&Z^\prime}
\end{equation}
where $\varpi$ and $\varpi^\prime$ are $G$-equivariant small resolutions such that 
$$
\mathrm{rk}\,\mathrm{Pic}(V)^G=\mathrm{rk}\,\mathrm{Cl}(V)^G=2=\mathrm{rk}\,\mathrm{Cl}(V^\prime)^G=\mathrm{rk}\,\mathrm{Pic}(V^\prime)^G,
$$
the map $\varsigma$ is a pseudo-automorphism that flops $\varpi$-contracted curves, 
both $\varphi$ and $\varphi^\prime$ are $G$-equivariant extremal contractions that can be of the following three types:
\begin{itemize}
\item a birational contraction,
\item a fibration into Del Pezzo surfaces over $\mathbb{P}^1$,
\item a conic bundle over a rational surface.
\end{itemize}
The diagram \eqref{eq:G-Sarkisov} is an example of a $G$-Sarkisov link (with $X$ being its \emph{center}).
When both $V$ and $V^\prime$ are smooth, such links have been studied in \cite{Takuchi1989,JahnkePeternellRadloff2005,JahnkePeternellRadloff2011,BlancLamy2012,CutroneMarshburn2013,ArapCutroneMarshburn2017,CutroneLimarziMarshburn2019,Takeuchi2022}.
Note that \eqref{eq:G-Sarkisov} is uniquely determined up to swapping its left and right sides.

If the morphism $\varphi$ in \eqref{eq:G-Sarkisov} is birational, then $Z$ is a Fano variety
with at most terminal singularities such that $\mathrm{rk}\,\mathrm{Cl}(Z)^G=1$,
so we are back to the case when $\mathrm{rk}\,\mathrm{Cl}^{G}(X)=1$ with $X$ replaced by $Z$.
Further, if the normalizer of $G$ in $\mathrm{Aut}(X)$ contains an automorphism $\sigma$ such that
$\mathrm{rk}\,\mathrm{Cl}^{\langle\sigma,G\rangle}(X)=1$, the diagram \eqref{eq:G-Sarkisov} simplifies as
\begin{equation}
\label{eq:G-Sarkisov-symmetric}
\xymatrix{
&&V\ar@{-->}[rrrr]^{\varsigma}\ar@{->}[dll]_{\varphi}\ar@{->}[d]^{\varpi}&&&&V\ar@{->}[d]_{\varpi}\ar@{->}[drr]^{\varphi}\\%
Z&&X\ar@{->}[rrrr]^{\sigma}&&&&X&& Z}
\end{equation}
In this case, we say that the $G$-Sarkisov link is symmetric.
For instance, if $X=X_3$ is the Segre cubic, this holds in many (but not all) cases.

\section{Burnside formalism}
\label{sect:incomp}

Here we explain a simplified version of the Burnside group formalism introduced in \cite{BnG}, 
which yields equivariant birational invariants of $G$-actions on algebraic varieties. We continue to work over an algebraically closed field $k$ of characteristic zero.

Applying equivariant blowups we may assume that the $G$-action is realized as a regular action on a {\em standard model} $(X,D)$ of the function field $K=k(X)$, i.e.,
\begin{itemize}
\item $X$ is smooth projective, $D$ a normal crossings divisor,
\item $G$ acts freely on $U:=X\setminus D$, 
\item for every $g\in G$ and every irreducible component $D$,
either $g(D) = D$ or $g(D)\cap  D =\emptyset$, 
\end{itemize}
see \cite[Section 7.2]{HKTsmall} for details. 
Given such a model, let 
$$
\{D_{\alpha}\}_{\alpha\in \mathcal A}
$$
be the set of irreducible divisors with nontrivial (thus necessarily cyclic) stabilizers $H_{\alpha}\subseteq G$;
we consider these up to conjugation in $G$. Each such $D_{\alpha}$ inherits a residual action of a group $Y_{\alpha}\subseteq Z_G(H_{\alpha})/H_{\alpha}$. 
Consider the subset $\mathcal A^{\mathrm{inc}}\subseteq  \mathcal A$ corresponding to
those divisors, together with the respective $Y_{\alpha}$-action, that cannot be obtained via equivariant blowups of {\em any} standard model of {\em any} $G$-variety. 

We have an assignment
\begin{equation}
\label{eqn:form}
X\actsfromright G \quad \Rightarrow   [X\actsfromright G]^{\mathrm{inc}}:= \quad \sum_{\alpha/\mathrm{conj}} (H_{\alpha}, Y_{\alpha} \actsfromleft k(D_{\alpha}), (b_{\alpha})),
\end{equation}
where the sum is over ($G$-conjugacy classes of) nontrivial cyclic $H_{\alpha}$, of symbols 
encoding
\begin{itemize} 
\item the stabilizer $H_{\alpha}$ of the generic point of $D_{\alpha}$, 
\item the residual action of $Y_{\alpha} \subseteq Z_G(H_{\alpha})/ H_{\alpha}$ on $D_{\alpha}$,
\item 
the character $b_{\alpha}$ of $H_{\alpha}$ in the normal bundle to $D_{\alpha}$ 
\end{itemize}
Note that $G$-conjugation extends to symbols in \eqref{eqn:form}, conjugating the $Y_{\alpha}$-action as well as the character $b_{\alpha}$, see \cite[Section 2]{KT-vector} for more details.

\begin{prop} \cite[Proposition 3.4]{KT-vector}
The class $[X\actsfromright G]^{\mathrm{inc}}$,
taking values in the free abelian group generated by symbols
\begin{equation}
\label{eqn:symbol}
(H, Y\actsfromleft k(D), (b)),\quad H\neq 1, 
\end{equation}
up to $G$-conjugation as above, 
is a well-defined $G$-birational invariant. 
\end{prop}

This is a rough invariant, obliterating information from nontrivial stabilizers in higher codimensions; but it already allowed to distinguish actions not accessible with previous methods \cite{KT-vector}.  
Given this, it becomes essential to provide a geometric 
characterization of incompressible divisorial classes. 
As explained in \cite[Section 3.6]{TYZ}, this property  {\em a priori} depends on the ambient group $G$. However, for some $Y$-actions on $D$, there is no such dependence, and we will call such symbols {\em absolutely incompressible}.  

For instance, by \cite[Proposition 3.6]{KT-vector}, in dimension $2$, a divisorial symbol \eqref{eqn:symbol} is absolutely incompressible iff: 
\begin{itemize}
\item
$D$ is a curve of genus $\ge 1$, or
\item 
$D$ is a curve of genus 0, and the residual $Y$ action on $D$ is not cyclic. 
\end{itemize}
In dimension $3$, sufficient conditions for incompressibility include: 
 \begin{itemize}
\item $D$ is not uniruled, 
\item $D$ is $Y$-birational to a $Y$-solid Pezzo surface,
\item the $Y$-action on $D$ has cohomology: $\rH^1(Y, \Pic(D)) \neq 0$,
\item the $Y$-action on $D$ is not equivariantly birational to a $Y$-action on a $\bP^1$-bundle over a curve.
\end{itemize}
If one is interested in comparing a $G$-action on a rational threefold to a linear action on $\bP^3$, one can exclude symbols \eqref{eqn:symbol} where $D$ admits a surjection onto a curve of genus $\ge 1$, as such symbols are not produced by the algorithm from \cite{KT-vector} which computes the class of  a linear action in the full Burnside group of \cite{BnG}, see \cite[Corollary 6.1]{TYZ}. 

Thus, for applications to linearizability in dimension 3, 
we need a classification of incompressible divisorial symbols of the form
\begin{equation}
\label{eqn:sym2}
(H, Y\actsfromleft k(\bP^2), (b)). 
\end{equation}
This can be obtained by combining classification schemes for $G$-surfaces, with the Burnside formalism of \cite{BnG}. 
There are two complementary approaches: via EMMP, as carried out in \cite[Section 8]{DI}, and using cohomology, as in \cite{prokhorovII}. 
The first approach allows to completely settle the linearization problem for rational surfaces \cite{sari}.
But in practice, the second approach is simpler to apply \cite{BogPro,KT-Cremona}.
For instance, if $X$ is a minimal $G$-Del Pezzo surface, then the following are equivalent, by \cite[Theorem 1.2]{prokhorovII}: 
\begin{itemize}
\item 
vanishing cohomology:
\begin{equation}
\label{eqn:h1}
\rH^1(G', \Pic(X)) = 0, \quad \text{ for all } G'\subseteq G. 
\end{equation}
\item no element of $G$ fixes a curve of genus $\ge 1$, 
\item either the degree of $X$ is at least $5$ or $G=C_3\rtimes C_4$ and 
$X$ is 
$G$-birational to a specific nonlinearizable $G$-Del Pezzo surface of degree 4. 
\end{itemize}
While not strictly necessary for the analysis of incompressible symbols, there is also a complete description of conic bundles satisfying \eqref{eqn:h1}, see \cite[Theorems 8.3 and 8.6]{prokhorovII}. 

\begin{prop}
\label{prop:P2-inc}
Let $Y\subset \PGL_3(k)$ be a finite nonabelian group, acting linearly on $D=\bP^2$.  
This action gives rise to an absolutely incompressible divisorial symbol in dimension 3, of the form \eqref{eqn:sym2},
if and only if the action is transitive. 
\end{prop}

\begin{proof} 
By \cite{sako}, if $Y$ acts transitively, then $D$ is $Y$-birationally rigid except for $Y=\fA_4$ or $\fS_4$. 
In particular, such actions are not birational to actions on Hirzebruch surfaces. 
If $Y=\fA_4$, it follows from \cite{Pinardine} or the proof of \cite[Proposition 43]{Ko2022}
that $D$ is not $Y$-birational to a Hirzebruch surface. 
Alternatively, one can notice that the Klein four subgroup of $\fA_4$ fixes a point in $D$,
while every faithful action of $\fA_4$ on a Hirzebruch surface does not enjoy this property. 
Same holds for $Y=\fS_4$.
Hence, if $Y$ acts transitively on $D$, then the symbol is absolutely incompressible.

Conversely, if $Y$ fixes a point on $\bP^2$, then a $Y$-equivariant blowup exhibits a Hirzebruch surface, and the symbol is compressible. 
\end{proof}

\begin{prop}[\cite{Pinardin}]
\label{prop:dp6}
Let $D$ be a Del Pezzo surface of degree 6 
and $Y\subset \Aut(D)$ a finite subgroup acting transitively on $(-1)$-curves.  If $Y\not\simeq C_6$ and $Y\not\simeq \fS_3$ then $D$ is $Y$-solid, and the $Y$-action is not (projectively) linearizable.
\end{prop}

In particular, the corresponding divisorial symbols in dimension 3, of the form \eqref{eqn:sym2} 
are absolutely incompressible.

\section{Linearizability of $G$-actions on tori}
\label{sect:tori}

Recall the structure of automorpisms $\Aut(T)$ of an algebraic torus $T=\mathbb G_m^n$, over a field $k$: there is an exact sequence of groups
$$
1\to T(k) \to \Aut(T) \stackrel{\phi}{\lra} \GL_n(\bZ) \to 1,
$$
and the homomorphism $\phi$ admits a section.
In particular, the torus $T$ admits automorphisms arising from finite subgroups $\Gamma\subset \GL_n(\bZ)$.

Let $X$ be a smooth projective $T$-equivariant compactification of $T$. Its Picard group has a presentation 
\begin{equation}
\label{eqn:seq}
0\to \mathfrak X^*(T) \to \mathrm{PL} \to \Pic(X)\to 0,
\end{equation}
where $\mathrm{PL}$ is the free abelian group spanned by irreducible components of the boundary $X\setminus T$, and $\mathfrak X^*(T)$ is the character group of $T$. In presence of $G$-actions, the sequence \eqref{eqn:seq} is a sequence of $\Gamma$-modules, where $\Gamma:=\phi(G)\subset \GL_n(\bZ)$; here
$\mathrm{PL}$ is a permutation module.

Lists of finite groups $\Gamma$, for small $n$, 
giving rise to actions on $\Pic(X)$ which 
satisfy {\bf(SP)}, and thus {\bf (H1)}, can be found in \cite{hoshi}. 
Linearizability properties of actions of finite subgroups of $\Aut(T)$ have been studied via birational rigidity techniques in \cite{Ch-toric,CSar}, where many examples of $G$-birationally rigid toric Fano threefolds were produced, and the groups $G$ considered typically had a large intersection with $T(k)$. 
The {\em stable} linearization problem of toric varieties, with $G$-actions satisfying $G\cap T(k)=\emptyset$, was settled in \cite[Proposition 12]{HT-torsor}. 

Linearization of actions on 2-dimensional tori is understood \cite{sari}.   
Let us recall the analysis in dimension 3, following \cite{kun3}: 

\

\noindent
{\em Step 1.} There are 4 maximal finite subgroups $\Gamma\subset \GL_3(\bZ)$, and in each case we can fix an explicit
(possibly singular) projective toric Fano threefold $X$ on which the $\Gamma$-action is regular: 
\begin{itemize}
\item[(F)] with $\Gamma_F:=C_2\times \fD_6$, acting on $X=\bP^1\times S$, where $S$ is a degree 6 Del Pezzo surface; 
\item[(C)] with $\Gamma_C:=C_2\times \fS_4$, acting on $X=\bP^1\times \bP^1\times \bP^1$;
\item[(S)] with $\Gamma_S:=C_2\times \fS_4$, acting on $X=X_{2,2}\subset \bP^5$, a singular intersection of two quadrics;
\item[(P)] with $\Gamma_P:=C_2\times \fS_4$; acting on the singular divisor
$$
X=\{ x_0y_0z_0t_0 = x_1y_1z_1t_1 \} \subset \bP^1\times \bP^1\times \bP^1\times \bP^1. 
$$
\end{itemize}
If $X$ is singular, we let $\widetilde{X}\to X$ be the equivariant blowup of its singular points. 
If $X$ is smooth, we let $\widetilde{X}=X$.

\

\noindent
{\em Step 2.}
By \cite[Proposition 1]{kun3}, in absence of an obstruction of type {\bf (SP)} for the $\Gamma$-module $\Pic(X)$, one of the following holds:
\begin{itemize}
\item[(a)] the $\Gamma$-module $\mathfrak X^*(T)$ splits and the action is birational to a product action,  
\item[(b)] $\mathfrak X^*(T)^{\Gamma}\neq 0$,
\item[(c)] the action is via a subgroup of $\Gamma_C$,
\item[(d)] the action is via $C_4$, $\fS_3$, or is linear, via a subgroup of $\fS_4$.
\end{itemize} 

\

\noindent
{\em Step 3.}
Fix the following subgroups in $\GL_3(\bZ)$:
$$
U_1:=\left\langle
\begin{pmatrix}
    -1&0&0\\
    0&0&-1\\
    0&-1&0
\end{pmatrix},
\begin{pmatrix}
    -1&-1&-1\\
    0&0&1\\
    0&1&0
\end{pmatrix}\right\rangle,
$$
$$
W_1:=\left\langle
\begin{pmatrix}
    0&1&1\\
    0&0&1\\
    -1&-1&-1
\end{pmatrix},
\begin{pmatrix}
    -1&0&0\\
    0&-1&0\\
    0&0&-1
\end{pmatrix}\right\rangle,
$$
$$
W_2:=\left\langle
\begin{pmatrix}
    0&0&1\\
    -1&-1&-1\\
    1&0&0
\end{pmatrix},
\begin{pmatrix}
    -1&-1&-1\\
    0&0&1\\
    0&1&0
\end{pmatrix},
\begin{pmatrix}
    -1&0&0\\
    0&-1&0\\
    0&0&-1
\end{pmatrix}\right\rangle.
$$
Then $U_1\simeq C_2\times C_2$, $W_1\simeq C_2\times C_4$, $W_2\simeq C_2^3$. It follows from \cite{kun3} that 
$$
\rH^1(U_1, \Pic(\widetilde{X}))\neq 0.
$$
Hence, if $\Gamma$ contains a subgroup conjugate to $U_1$, then $\Gamma$ and $G$ do not satisfy {\bf(H1)}.
There are exactly 12 conjugacy classes of such subgroups $\Gamma\subset \GL_3(\bZ)$. 

Furthermore, if $\Gamma$ contains a subgroup conjugate to $W_1$ or  $W_2$, then 
it also follows from \cite{kun3} that the $\Gamma$ and $G$-action on $\mathrm{Pic}(\widetilde{X})$ do not satisfy {\bf(SP)} 
(but $W_1$ and $W_2$ do satisfy  {\bf(H1)}).

\

Here we adapt this to the study of linearizability of the actions of these groups on tori, via a case by case study as in {\em Step 1}.
We identify subgroups of $\GL_3(\bZ)$ with subgroups of $\mathrm{Aut}(X)$ via the standard lifts to $\mathrm{Aut}(T)$ with fixed point $(1,1,1)\in T$. 

\

\noindent
{\bf Case (F):}
Assume that $\Gamma\subseteq\fD_6$, with trivial action on $\bP^1$.
The action of $\fD_6$ on $S$ is not linearizable by Proposition~\ref{prop:dp6} or \cite{isk-s3}. However, by \cite[Proposition 12]{BBT}, $\bP^1\times S$ is {\em linearizable} with the trivial action on $\bP^1$. 
The action of every proper subgroup of $\fD_6$ on $S$ is linearizable. 

Conversely, if $\Gamma=\Gamma_F$, then the action on the toric threefold contributes an absolutely incompressible,
by Proposition~\ref{prop:dp6},
symbol
$$
(C_2, \fD_6\actsfromleft k(S), (1)),
$$
from the origin in the torus.
On the other hand, such symbols do not arise from projectively linear actions, as $S$ is not $\fD_6$-linearizable, nor birational to a product of (projectively) linear actions \cite[Example 9.2]{TYZ-3}.

\

\noindent
{\bf Case (C):}  
The action of $\Gamma_C$ on $\bG_m^3$ is generated by
$$
(x_1,x_2,x_3)\mapsto (x_3,x_2,\frac{1}{x_1}),\quad (\frac{1}{x_1},\frac{1}{x_3},\frac{1}{x_2}), \quad (\frac{1}{x_1},\frac{1}{x_2},\frac{1}{x_3}). 
$$
A birational change of coordinates $y_i:=\frac{1-x_i}{1+x_i}$ yields the action
$$
(y_1,y_2,y_3)\mapsto (y_3,y_2,-y_1),\quad (-y_1,-y_3,-y_2), \quad  (-{y_1},-{y_2},-{y_3}), 
$$
which is clearly linearizable.

\

\noindent
{\bf Case (S):}  
By results in Section~\ref{sect:segre} (Proposition~\ref{prop:x3plane} and Theorem~\ref{thm:segrelin}), it suffices to establish 
the linearizability of the action of 
 $\Gamma=C_2^2$ on $\bG_m^3$, with coordinates $x_1,x_2,x_3$, via
$$
\sigma: x_1\leftrightarrow x_3,\,\, x_2\mapsto 1/x_1x_2x_3, \quad \tau: x_j\mapsto 1/x_j, \quad j=1,2,3. 
$$
This action is conjugate to a subgroup of $\Gamma_C$ in Case (C) and is linearizable. 

\

\noindent
    {\bf Case (P):}  For $\Gamma\subseteq \Gamma_P$, only five groups $G$ do not appear in {Case (F), (C)} or {(S)}, up to conjugation:
$$
\Gamma=\fA_4, C_2\times\fA_4, \fS_4,\fS_4',\,\,\text{or}\,\, \Gamma_P.
$$
Each of them contains a subgroup conjugate to $U_1$ and thus is not linearizable. In summary, a subgroup of $\Gamma_P$ is not linearizable if and only if it contains one of $U_1,W_1,$ or $W_2$.

\

We summarize the above discussion:

\begin{theo} 
\label{thm:lin-tori}
Let $T=\bG_m^3$ and $G\subset \Aut(T)$ be such that 
$\Gamma:=\phi(G)$ contains $U_1,W_1$, or $W_2$. Then the $G$-action on $T$ is not stably linearizable. 
Assume that $T^G\ne\emptyset$, i.e., $G$ fixes a point in $T$. Then 
\begin{itemize}
\item 
if $\Gamma=C_2\times \fD_6$, then the action is not linearizable but stably linearizable,
\item 
if $\Gamma\neq C_2\times \fD_6$ and does not contain $U_1,W_1$, or $W_2$, then the action is linearizable. 
\end{itemize}
\end{theo}

\begin{proof}
    If $\phi(G)=U_1$, then $G$ has an $\mathbf{(H1)}$-obstruction to stable linearizability. If $\phi(G)=W_1$ or $W_2$, then $G$ has an $\mathbf{(SP)}$-obstruction to   stable linearizability. When $T^G\ne\emptyset$, we can assume $(1,1,1)\in T^G$ up to translation. Then we are in one of the cases (F), (C), (S) or (P) discussed above.

    Stable linearizability of the $C_2\times \fD_6$-action is established as in \cite{HT-torsor}, using the equivariant version of the torsor formalism. 
\end{proof}

\begin{rema}
    The second part of Theorem~\ref{thm:lin-tori} does not hold without the assumption $T^G\ne\emptyset$. For example, consider $\Gamma=C_2^2\subset \GL_3(\bZ)$ from Case (S) above. Up to conjugation, we find two translation-free lifts $G\subseteq \Aut(T)$ of $\Gamma$, i.e., $\phi(G)=\Gamma$: the standard lift generated by 
$$
\sigma: y_1\mapsto \frac{1}{y_1},\,\, y_2\leftrightarrow y_3, \quad\tau: y_j\mapsto\frac{1}{y_j}, \quad j=1,2,3,
$$
and a twist of it generated by 
$$
\sigma': y_1\mapsto -\frac{1}{y_1},\,\, y_2\leftrightarrow y_3, \quad\tau: y_j\mapsto\frac{1}{y_j}, \quad j=1,2,3.
$$
The standard lift is linearizable. The twisted one is not linearizable as the $G$-action on the projective model $\bP^1\times\bP^1\times\bP^1$ does not have a fixed point. In particular, these two lifts are not equivariantly birational.

    
\end{rema}

\section{Geometry of the Segre cubic}
\label{sect:segre}

Rationality of forms of the Segre cubic threefold over nonclosed fields has been considered in \cite{FR}; there exist nonrational forms over nonclosed fields. All forms over the reals are rational \cite[Corollary 2.5]{avilov-forms}.

There are 55 nontrivial conjugacy classes of subgroups of $\fS_6$. 
By \cite{Avilov,ChS-5}, everything is known in the {\em minimal case}, when  
$$
G=\fA_5, \quad \fS_5, \quad \fA_6, \quad \fS_6. 
$$
Namely, there are two $\fA_5$ and $\fS_5$ classes, corresponding to the {\em standard}, respectively, {\em nonstandard} embedding of these groups into $\fS_6$. 
If $G$ is a standard subgroup $\fA_5$, then $\mathrm{rk}\,\mathrm{Cl}(X_3)^G=1$ and $X_3$ is $G$-birationally super-rigid \cite{Avilov}.
This also implies that $X_3$ is $G$-birationally super-rigid if $G$ is a standard $\fS_5$, $\fA_6$ or the whole group $\fS_6$,
and the actions of these groups are not linearizable. 
Vice versa, if $G$ is a nonstandard subgroup $\fS_5$,
then  $\mathrm{rk}\,\mathrm{Cl}(X_3)^G=2$, and we have the following nonsymmetric $G$-Sarkisov link:
$$
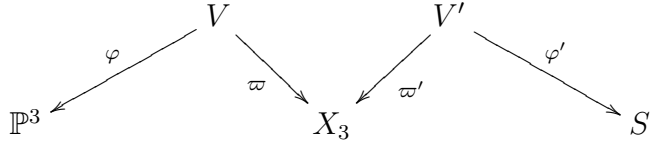
\xymatrix{
&&V\ar@{->}[dll]_{\varphi}\ar@{->}[dr]_{\varpi}&&V^\prime\ar@{->}[dl]^{\varpi^\prime}\ar@{->}[drr]^{\varphi^\prime}\\%
\mathbb{P}^3&&&X_3&&& S}
$$
where $\varphi$ is a blow up of a $G$-orbit of length $5$,
both $\varpi$ and $\varpi^\prime$ are flopping contractions,
$S$ is a smooth Del Pezzo surface of degree $5$, 
and $\varphi^\prime$ is a $\mathbb{P}^1$-bundle.
In particular, the actions of all subgroups of the nonstandard subgroup $\fS_5$ are linearizable, 
e.g., the nonstandard subgroup $\fA_5$, the unique subgroup $C_5, \mathfrak D_5, \mathfrak F_5$.  

Next we exclude $G$-actions with $G$-fixed singular points, since then $X_3$ is $G$-birational to $\bP^3$. 
There are 25 such conjugacy classes; all such $G$ are contained in the unique class of $\fS_3^2\rtimes C_2$. 

Therefore, all subgroups of the nonstandard $\fS_5'$ or 
$\fS_3^2\rtimes C_2$  yield linearizable actions.
The remaining groups are contained in one of two nonconjugate $C_2\times \fS_4$; 
one of them preserve a plane $\Pi\subset X_3$ (this is a {\em nonstandard} subgroup of $\fS_6$), 
and another one does not preserve any plane in $X_3$ (this is a {\em standard} subgroup, conjugate to the group generated by the involution $(12)$ and permutations of the remaining indices). 

\begin{prop}
\label{prop:x3plane}
Suppose that $X_3$ contains a $G$-invariant plane $\Pi$. Then there exists the following $G$-equivariant diagram:
$$
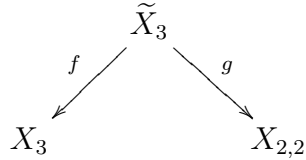
\xymatrix{
&\widetilde{X}_3\ar[ld]_{f}\ar[rd]^{g}&\\
X_3&&X_{2,2}}
$$
where 
\begin{itemize}
\item $X_{2,2}\subset \bP^5$ is the unique singular toric complete intersections of two quadrics with six nodes,
\item 
$f$ is a small birational morphism, and
\item 
$g$ is a blowdown of the proper transform of the plane $\Pi$ to a smooth point of $X_{2,2}$.
\end{itemize}
Moreover, 
the $G$-action on $X_{2,2}$ preserves the torus in $X_{2,2}$, and $G$ fixes a point in the torus.  
\end{prop}
\begin{proof}
  Unprojecting from the $G$-invariant plane, we obtain the required commutative diagram.
\end{proof}

The following diagram summarizes the relations between $G$-actions on $X_3$, when $G$ is contained in one of two nonconjugate subgroups $C_2\times \fS_4$ and is not contained in the nonstandard $\fS_5'$ or $\fS_3^2\rtimes C_2$:

\

{\tiny

\centerline{
\xymatrix{
 \fS_4,\fS_4' \ar[d]& C_2\times \fS_4 (\text{no } \Pi) \ar[l]\ar[d]\ar[dr] & &  \ar[dl] C_2\times \fS_4 (\Pi)  \ar[d] \ar[r] &    \fS_4 \ar[d]    \\
  \fA_4 \ar[ddr]  &  \ar[l] C_2 \times \fA_4 & C_2\times \fD_4   \ar[d] \ar[dl] \ar[dr]\ar[drr]  & C_2 \times \fA_4 \ar[dr]  &{\color{red}U_1=C_2^2} & \\
      & \fD_4,\fD_4'\ar[d]  &   C_2^3\ar[dl]\ar[dr]  &{\color{red}W_2=C_2^3}\ar[d]&{\color{red}W_1=C_2\times C_4} \ar[dl]  \\
                      &    {\color{red}U_1=C_2^2}  & &   {\color{blue}C_2^2}     &   &        
}
}
}

\

\noindent
By Proposition~\ref{prop:x3plane}, we can identify subgroups of $\Aut(X_3)$ leaving a plane invariant with subgroups of $\Aut(X_{2,2})$ fixing the origin of the torus,
which also can be identified with finite subgroups of $\GL_3(\bZ)$ as in Section~\ref{sect:tori}.
In this way, we identified the subgroups $C_2^2, C_2^3$ and $C_2\times C_4$ in the diagram with subgroups of $\GL_3(\bZ)$, using notations from Section~\ref{sect:tori}:
\begin{itemize}
\item 
$U_1=C_2^2\subset \GL_3(\bZ)$ is the group with {\bf (H1)} obstruction to stable linearizability,
\item 
$W_1=C_2\times C_4,W_2=C_2^3$ are groups with {\bf (SP)} obstruction to stable linearizability. 
\end{itemize}
Moreover, the other subgroup $C_2^2$ in the diagram can be uniquely characterized by the following  geometric conditions: \begin{enumerate}
    \item $\mathrm{Sing}(X)$ splits as a union of five $C_2^2$-orbits of length $2$, and
    \item it is not contained in the nonstandard $\fS_5'$, and
    \item it leaves exactly three planes in $X_3$ invariant.
\end{enumerate}

\begin{coro}
\label{coro:x3noplane}
Suppose that the following conditions are satisfied:
\begin{enumerate}
\item $G$ does not fix a singular point of $X_3$, and
\item $G$ is not contained in the nonstandard $\fS_5'$, and
\item $X_3$ does not contain $G$-invariant planes.
\end{enumerate} 
Then $G$ does not satisfy {\bf (H1)}, and the $G$-action is not stably linearizable.
\end{coro}

\begin{proof}
By the diagram above, $G$ contains the subgroup $C_2^2$ with the {\bf(H1)} obstruction. This can also be checked directly, via {\tt Magma}.
\end{proof}

If $G$ leaves invariant a plane $\Pi\subset X_3$, the linearization problem is reduced 
to $G$-actions on a three-dimensional torus with $G$-fixed points, which was studied in \cite[Section 9]{HT-quad}, and in detail in Section~\ref{sect:tori}. 
Combining Proposition~\ref{prop:x3plane}, Corollary~\ref{coro:x3noplane} and results in Section~\ref{sect:tori}, we obtain:

\begin{theo}
\label{thm:segrelin}
The $G$-action on $X_3$ is linearizable if and only if
either
\begin{itemize}
\item $G$ fixes a singular point on $X_3$, or
\item $G$ is contained in the nonstandard $\fS_5'$, or
\item $G=C_2^2$, $X_3$ contains three $G$-invariant planes, and $\mathrm{Sing}(X)$ splits as a union of five $C_2^2$-orbits of length $2$.
\end{itemize}
Moreover, when the $G$-action is not linearizable, it is not stably linearizable.
\end{theo}

\begin{proof}
   If one of the first two conditions is satisfied, then $G$ is linearizable, as explained above. If $X_3$ does not contain $G$-invariant planes, then $G$ is not stably linearizable by Corollary~\ref{coro:x3noplane}. Hence we may assume $G$ does not fix a singular point, $G$ is not contained in the nonstandard $\fS_5'$ and $X_3$ contains a $G$-invariant plane $\Pi$. By Proposition~\ref{prop:x3plane}, there then exists a $G$-equivariant birational map from $X_3$ to the toric intersection of two quadrics $X_{2,2}\subset \bP^5$. 

   Going through the group diagram above, we see $G$ is not stably linearizable when $G$ is not conjugate to the $C_2^2$ characterized in the third condition; this $C_2^2$ can be identified with the subgroup in $\GL_3(\bZ)$ generated by
$$
\begin{pmatrix}
    0&0&1\\
    -1&-1&-1\\
    1&0&0
\end{pmatrix},
\begin{pmatrix}
    -1&0&0\\
    0&-1&0\\
    0&0&-1
\end{pmatrix}.
$$
It is linearizable, as explained in Section~\ref{sect:tori}.
\end{proof}

\begin{rema}
The Burnside formalism of \cite{BnG} {\em does not} allow to decide the linearizability of the actions of both $C_2^2$ at the bottom of the lattice diagram above. On the other hand, the formalism of {\em incompressible} symbols as in Section~\ref{sect:incomp}, proves nonlinearizability in several cases; note that these cases are obstructed by {\bf(SP)}, as they contain $W_1$.  

    Let $G=C_2\times\fA_4$, generated by 
    \begin{align*}
    \iota: (x_1,x_2,x_3,x_4,x_5,x_6)\mapsto (x_3,x_5,x_1,x_6,x_2,x_3),\\
    \tau: (x_1,x_2,x_3,x_4,x_5,x_6)\mapsto (x_3,x_5,x_1,x_4,x_2,x_6),\\
    \sigma: (x_1,x_2,x_3,x_4,x_5,x_6)\mapsto (x_5,x_6,x_2,x_1,x_4,x_3).
    \end{align*}
The fixed locus for the involution $\iota$ is a plane $\Pi\subset X_3$ given by 
$$
x_1+x_3=x_2+x_5=x_4+x_6=0
$$
with a residue $\fA_4$-action on it. This produces an absolutely incompressible divisorial symbol (see Proposition~\ref{prop:P2-inc})
$$
\mathfrak s:=(C_2,\fA_4\actsfromleft k(\bP^2), (1)).
$$
The model $X_3\actsfromright G$ is not in standard form. However, $G$ does not leave invariant any irreducible subvariety of $X_3$ with nontrivial stabilizer except $\Pi$. 
Therefore, no equivariant blow-up of $X$ can possibly contribute the symbol $\mathfrak s$ to the class $[X_3\actsfromright G]^{\rm inc}$. The class then contains the incompressible symbol $\mathfrak s$ with multiplicity $1$. On the other hand, the algorithm in \cite{BnG}, implemented in \cite{TYZ}, shows that $[\mathbb P^3\actsfromright G]^{\rm inc}$ contains $\mathfrak s$ with multiplicity $2$ for any (projectively) linear action $\mathbb P^3\actsfromright G$. We conclude that this $G$-action  on $X_3$, and thus also the action of $C_2\times\fS_4$ containing $G$, are  not (projectively) linearizable.
\end{rema}

\section{Geometry of the Burkhardt quartic}
\label{sect:burk}

The Burkhardt quartic $X_4$ can be defined in $\bP^4\subset \bP^5$ by \eqref{eqn:quart},
or in $\bP^4$ by \eqref{eqn:burg}. Up to projectivity, $X_4$ is the unique quartic threefold with 45 nodes \cite{dejong},
and $\mathrm{Aut}(X_4)=\mathsf{PSp}_4(\bF_3)$ acts on $\bP^4$ via an irreducible 5-dimensional representation of its central extension~$\mathrm{Sp}(\bF_3)$.
Our goal is to identify subgroups $G\subseteq \mathsf{PSp}_4(\bF_3)$ giving rise to (projectively) linearizable actions on $X_4$. 

Arithmetic aspects of the Burkhardt quartic, in particular, its rationality over nonclosed ground fields $k$, have been explored in \cite{bruin-1,bruin-2,Calegari}. For example, the form  \eqref{eqn:burg} is rational over $\bQ$. For all forms $X$ of $X_4$ over nonclosed fields of characteristic zero there exist a dominant, degree 6, map $M\to X$, 
where $M$ is a Brauer-Severi variety of dimension 3 \cite[Theorem 1.1]{bruin-1}; 
forms arising from moduli spaces of abelian surfaces are unirational, 
in particular, their rational points are Zariski dense 
(see  \cite{Calegari} and references therein). 
It is an open problem to determine which forms $X$ are rational over $\bQ$ \cite[Question 2.9]{bruin-2};  
there certainly are $k$-forms that are not $k$-rational \cite{Calegari,bruin-1}. 

Note that \cite[Section 3]{Calegari} lists all subgroups $G\subseteq \mathsf{PSp}_4(\bF_3)$ with nontrivial cohomology 
$$
\rH^1(G, \Pic(\widetilde{X}_4)), 
$$
where $\widetilde{X}_4$ is the standard resolution of singularities of the Burkhardt quartic, denoted by 
$\mathcal A^*_2(3)$ in \cite{Calegari}. 
Recall that this is an obstruction to stable (projective) linearizability of the $G$-action. 
In particular, of the 115 conjugacy classes of nontrivial subgroups of $\mathsf{PSp}_4(\bF_3)$, 
only 26 do not have the cohomological obstruction to {\bf (H1)}. 

\begin{rema}
\label{rema:E6}
Note that 
$\mathsf{PSp}_4(\bF_3)\subset W(\mathsf E_6)$, as an index-two subgroup. If we consider the $W(\mathsf E_6)$-action on the Picard lattice of a smooth cubic surface, then, by \cite{TY}, {\em every} subgroup 
$G\subseteq W(\mathsf E_6)$ arising from a {\em minimal} 
action on a cubic surface
has nontrivial cohomology. 
\end{rema}

\begin{prop}
\label{prop:burk}
Assume that $G$ contains an involution exchanging two coordinates in $\mathbb{P}^5$. Then the $G$-action is not projectively linearizable and not equivariantly birational to the action on the Segre cubic. 
\end{prop}

\begin{proof}
We apply the formalism of Section~\ref{sect:incomp}. 
The involution action leads to classes 
$$
(C_2, Y\actsfromleft k(D), (1)),
$$
where $D$ is a quartic K3 surface with 12 nodes. This is an absolutely incompressible symbol, since $D$ is not uniruled. Furthermore,  
it does not arise from projectively linear actions.

On the other hand, every $C_2$-action on the Segre cubic fixes a singular point, and the action is linearizable.
\end{proof}

\begin{rema}
\label{rem:bbur}
There are 12 (conjugacy classes of) subgroups of $\mathsf{PSp}_4(\bF_3)$ containing this involution and satisfying {\bf (H1)}:
$$
C_2,C_4,C_2^2,C_6,C_6',Q_8,\fD_4,C_2\times C_6, C_{12},\SL_2(\bF_3), C_3\times Q_8, C_3\rtimes \fD_4,
$$
in particular, cohomology does not allow to distinguish
linearizability from nonlinearizability in these cases; the corresponding actions are specified in \cite{CTZ-tables}. 
The Burnside obstruction to linearizability in Proposition~\ref{prop:burk} vanishes for $X_4\times \bP^1$, with trivial action on the second factor, see, e.g.,  \cite[Section 3.5]{TYZ-3}. The {\bf (SP)}-obstruction is also trivial, at least for $G=C_2$. 
Thus we are led to speculate that the $C_2$-action on the threefold $X_4\times \bP^1$ {\em is} linearizable. 
\end{rema}

Using Proposition~\ref{prop:inv-class}, we see that \mbox{$\mathrm{rk}\,\mathrm{Cl}^G(X_4)=1$}
for $34$ out of $115$ conjugacy classes of nontrivial subgroups in $\mathsf{PSp}_4(\bF_3)$.
All these groups are not linearizable by the following result.

\begin{prop}
\label{thm:Burkhard-BSR}
Let $G\subseteq\mathrm{Aut}(X_4)$ be such that \mbox{$\mathrm{rk}\,\mathrm{Cl}(X_4)^G=1$}.
Then $X_4$ is $G$-birationally super-rigid.
\end{prop}

\begin{proof}
See Section~\ref{sect:Burk-rigid}.    
\end{proof}

Now, excluding groups with nontrivial cohomological obstructions to linearizability, those containing an involution fixing a K3 surface  in~$X_4$,
and those with $\mathrm{rk}\,\mathrm{Cl}^G(X_4)=1$,
we are left with the following tree of 12 groups, where the left column lists $\mathrm{rk}\, \mathrm{Cl}(X_4)^G$:

\ 

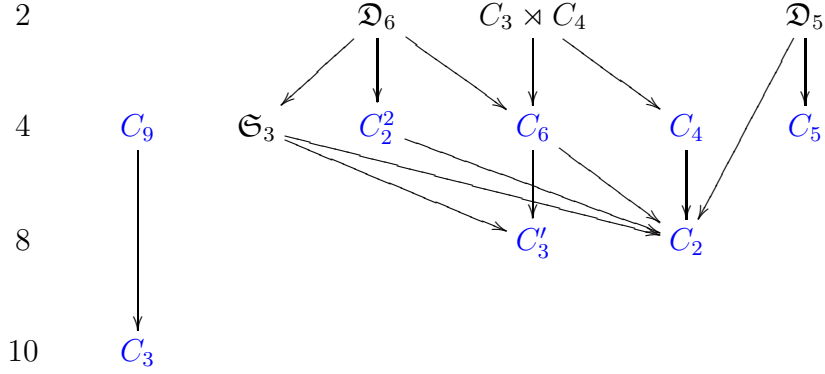
$$
\xymatrix{
2  &             &  & \fD_6\ar[dl]  \ar[d] \ar[dr]         & C_3\rtimes C_4  \ar[d] \ar[dr]  &                       &  \fD_5 \ar[d]\ar[ddl] \\
4  &  {\color{blue}C_9}\ar[dd] & \fS_3\ar[drrr]\ar[drr] & {\color{blue}C_2^2}\ar[drr] &     {\color{blue}C_6} \ar[dr]\ar[d]                 & {\color{blue}C_4} \ar[d]            & {\color{blue}  C_5}                   \\
8  &             &  &                 & {\color{blue} C_3'}                                      &   {\color{blue} C_2}                 &                     \\
10 &  {\color{blue}C_3}        &  &                     &                                       &                       &       
}
$$

\

Concretely, put

{\tiny
$$
\sigma_4=\begin{pmatrix}
             1&0&0&0&0\\
         0&0&0&0&q\\
         0&0&0&q^2&0\\
         0&1&0&0&0\\
         0&0&1&0&0\\
\end{pmatrix},\quad\sigma_9=\begin{pmatrix}
 2q + 1&  2q + 4     &        0   &          0     &        0\\
-2q - 1 &   q + 2     &        0   &          0    &         0\\
            0        &     0&  2q + 1 &  -q + 1 & 2q + 1\\
            0        &     0&   -q + 1 & 2q + 1 & 2q + 1\\
            0        &     0&   -q - 2 &  -q - 2&  2q + 1\\
\end{pmatrix},
$$
$$
\sigma_2=\begin{pmatrix}
                         1 &          0 &          0 &          0 &          0\\
          0 &          0 &          0 &          0 &          1\\
          0 &          0 &          0 &q^2 &          0\\
          0 &          0 &     q &          0 &          0\\
          0 &          1 &          0 &          0 &          0\\
\end{pmatrix},\quad\quad \quad\quad \quad\quad \quad \quad \,\,\,
\sigma_6=\begin{pmatrix}
            -1  &   2q &  2q^2 &           2&            2\\
           1  &     q &    q^2 &           1&             2\\
   q^2  &           2 &      q &    q^2&     q^2\\
    q^2  &          1 &     2q &   q^2&     q^2\\
   q^2  &          1 &      q &2q^2&    q + 1\\
\end{pmatrix},
$$
$$
\sigma_3=\begin{pmatrix}
               -1 &            2 &     2q &2q^2  &           2\\
       q &       q & -2q^2 &            1  &      q\\
            1 &           -2 &       q &  q^2  &           1\\
            1 &            1 &       q &  q^2  &          -2\\
  q^2 &  q^2 &            1 &    -2q  & q^2\\
\end{pmatrix},\quad \quad \quad \quad \quad \,\,\,\sigma_5=\begin{pmatrix}
               -1  &  2q &2q^2  &         2 &          2\\
  q^2  &         1 &    2q &  q^2 &  q^2\\
     q  & q^2 &          1&     2q &     q\\
     q  & q^2 &          1&      q &    2q\\
  q^2  &          2 &     q&   q^2 &  q^2\\
\end{pmatrix},
$$
}

\noindent
where $q$ is a primitive third root of unity.
The groups in the diagram are given in the above generators by: 
$$
\fD_6=\langle\sigma_2,\sigma_6\rangle, \quad C_3\rtimes C_4=\langle\sigma_3,\sigma_4\rangle,\quad \fD_5=\langle\sigma_4^2,\sigma_5\rangle,\quad C_9=\langle\sigma_9\rangle,
$$
$$
\fS_3=\langle\sigma_2,\sigma_6^2\rangle, \quad C_2^2=\langle\sigma_2,\sigma_6^3\rangle,\quad C_6=\langle\sigma_6\rangle,\quad C_4=\langle\sigma_4\rangle,
$$
$$
C_5=\langle\sigma_5\rangle,\quad C_3=\langle\sigma_9^3\rangle,\quad
C_3'=\langle\sigma_3\rangle,\quad C_2=\langle\sigma_4^2\rangle,
$$

The rest of this section is devoted to a case by case analysis of these actions, organized as follows: 
First, we recall classical linearization constructions for $C_9$ and $C_5$. 
Then we present a new linearization construction for $C_6$,
which also gives linearization of $C_2$ and $C_3$,  
and use this construction to create a $\fD_6$-equivariant birational map from $X_4$ to a smooth quadric threefold $X_2\subset\mathbb{P}^4$,
which gives a linearization for $C_2^2$.
Finally, we present a linearization construction for $C_4$.
The remaining four subgroups are 
\begin{center}
$\mathfrak{S}_3$, $\fD_5$, $\fD_6$, $C_3\rtimes C_4$.
\end{center}
We do not know whether or not they are linearizable.

To study the linearization problem for $\fD_6$ and $\mathfrak{S}_3$,
one can use their actions on the quadric $X_2$.
For $\fD_5$, recall from \cite[Example 5.8]{CKS} that there exists 
a $\fD_5$-equivariant commutative diagram \eqref{eq:G-Sarkisov-symmetric} such that $Z$ is the smooth Del Pezzo surface of degree $5$, and $\varphi$ is a conic bundle, whose discriminant curve is the union of all $(-1)$-curves in $Z$.
This might be a good model for the  study of the linearization problem for $\fD_5$. 
Before showing linearization constructions for abelian groups, we present an explicit $C_3\rtimes C_4$-equivariant birational map from $X_4$ to a fibration into quartic Del Pezzo surfaces, which could hopefully be used to study the linearization problem $C_3\rtimes C_4$.

\begin{exam}
\label{exam:C3-C4}
Let $G=\langle \sigma_3, \sigma_4\rangle\simeq C_3\rtimes C_4$. The defining 
equation of $X_4\subset\bP^4$ can be rewritten as $f_2r_2=g_2h_2$, where
{\footnotesize
\begin{align*}
    f_2= &y_1y_2+q^2y_1y_3-y_1y_4-q^2y_1y_5-q^2y_2^2+qy_2y_3+(q-1)y_2y_4+\\
         &(-q-2)y_2y_5-y_3^2+(-q-2)y_3y_4+(-2q-1)y_3y_5+q^2y_4^2-qy_4y_5+y_5^2,\\
    g_2= &y_1^2 -q^2y_1y_4 - qy_1y_5 - qy_2^2 + y_2y_3 -q^2y_2y_4 -q^2y_2y_5-q^2y_3^2-q^2y_3y_4-y_3y_5,\\
r_2= &qy_1y_4 + y_1y_5 + y_2^2 -q^2y_2y_3 + qy_3^2,\\
h_2= &y_1^2 + q^2y_1y_4 + qy_1y_5 + qy_2^2 - y_2y_3 - y_2y_4 -qy_2y_5+ q^2y_3^2- qy_3y_4 - qy_3y_5.
\end{align*}}
\noindent The surfaces $\{f_2=h_2=0\}$ and $\{g_2=r_2=0\}$ generate a $G$-invariant pencil,
and there exists a $G$-equivariant birational morphism $\widetilde{X}_4\to X_4$,
where 
$$
\widetilde{X}_4=\{f_2u+g_2v=h_2u+r_2v=0\}\subset\bP^4\times\bP^1_{u,v},
$$
and the $G$-action on  $\bP^4\times\bP^1$ is generated by
$$
    \sigma_4\times\begin{pmatrix}
      1&0\\
      q^2&-1
    \end{pmatrix}_,\quad
    \sigma_3\times\begin{pmatrix}
      1&-2q-1\\
      0&q
\end{pmatrix}_.
$$
The birational morphism $\widetilde{X}_4\to X_4$ is given by the projection to the first factor of $\bP^4\times\bP^1$,
while the projection to the second factor $\widetilde{X}_4\to\mathbb{P}^1$ is a $G$-equivariant fibration into Del Pezzo surfaces of degree $4$. 
\end{exam}

\subsection*{Linearization of $C_9$}
By \cite[Remark 4.3] {bruin-2}, the classical  parametrization of the Burkhardt quartic by 
Baker is $C_9$-equivariant, i.e., the action of $C_9$ pulls back to an action on $\bP^3$.  
An explicit description of the base locus of the birational map $\phi:\bP^3 \dashrightarrow X_4$ can be found in \cite{finkelberg}:

Consider a configuration in $\bP^3$ consisting of 9 lines $l_1,...,l_9$ with $l_i$ meeting $l_{i+1}$ in a point $p_i$, and $l_9$ meeting $l_1$ in $p_9$. The points $\{p_1,p_4,p_7\}, 
\{p_2,p_5,p_8\}$ and 
$\{p_3,p_6,p_9\}$ define lines that intersect in a common point $p_{10}$ and 
$$
l_1\cap l_4\cap l_7, \quad l_2 \cap l_5\cap l_8, \quad l_3\cap l_6\cap l_9
$$ 
define a further 3 points. There is a unique such configuration, modulo $\PGL_4$, 
and the linear system of all quartic surfaces in $\bP^3$ containing these $9$ lines gives $\phi$.
The symmetry group of this configuration is indeed $C_9$.

\subsection*{Linearization of $C_5$}
There is also another parametrization, due to Todd \cite{todd}, and described in detail in \cite[Section 5.1]{pettersen}. There is a rigid configuration of 10 lines and 15 points in $\bP^3$, and the linear system of quartic surfaces passing through these 10 lines gives a birational map from $\bP^3$ to the Burkhardt quartic \cite[Figure 5.1]{pettersen}. For example, the $10$ lines can be given by equations and visualized as follows
\begin{align*}
l_1&=\{x_0 - x_3=x_1 + q^2x_2 + qx_3=0\}, &l_2=\{x_2=x_3=0\},\\
l_3&=\{x_1=x_2 + q^2x_3=0\}, &l_4=\{x_1=x_2=x_3\},\\
l_5&=\{x_0 + q^2x_3=x_2 + q^2x_3\}, &l_6=\{x_0=x_1=x_2\},\\
l_7&=\{x_0 + q^2x_3=x_1 - x_3\},&l_8=\{x_0=x_1=0\},\\
l_9&=\{x_0 + q^2x_1=x_2\}, &l_{10}=\{x_0=x_3=0\}.
\end{align*}

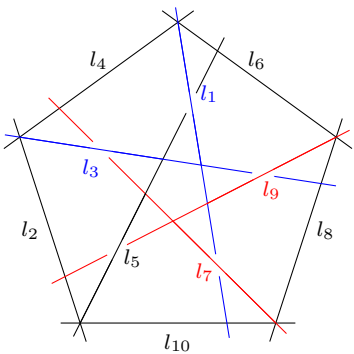
{\begin{center}
\begin{tikzpicture}[scale=1.3, dot/.style={circle,inner sep=1pt,fill,label={#1},name=#1},
 extended line/.style={shorten >=-#1,shorten <=-#1},
 extended line/.default=1cm]
 
\textbf{}
\coordinate (p1) at (-1.68,0);
\coordinate (p2) at ($(p1)+(36:2cm)$);
\coordinate (p3) at ($(p2)+(360-36:2cm)$);
\coordinate (p4) at ($(p3)+(180+72:2cm)$);
\coordinate (p5) at ($(p4)+(180:2cm)$);
\coordinate (p1m) at  ($(p1)+(180+36:0.2cm)$);
\coordinate (p2m) at  ($(p2)+(180-36:0.2cm)$);
\coordinate (p3m) at  ($(p3)+(72:0.2cm)$);
\coordinate (p4m) at  ($(p4)+(0:0.2cm)$);
\coordinate (p5m) at  ($(p5)+(360-72:0.2cm)$);
\coordinate (p1p) at  ($(p1)+(36:0.5cm)$);
\coordinate (p2p) at  ($(p2)+(360-36:0.5cm)$);
\coordinate (p3p) at  ($(p3)+(180+72:0.5cm)$);
\coordinate (p4p) at  ($(p4)+(180:0.5cm)$);
\coordinate (p5p) at  ($(p5)+(108:0.5cm)$);

\draw [name path=l1](p1m) --node[above] {\tiny $l_4$} ++ (36:2.4cm) ;
\draw [name path=l2](p2m) -- node[above] {\tiny $l_6$}++ (360-36:2.4cm) ;
\draw [name path=l3](p3m) -- node[right] {\tiny $l_8$}++ (180+72:2.4cm) ;
\draw [name path=l4](p4m) --node[below] {\tiny $l_{10}$} ++ (180:2.4cm) ;
\draw [name path=l5](p5m) --node[left] {\tiny $l_2$} ++ (108:2.4cm) ;
\draw [name path=k4,blue][extended line=0.2cm] (p4p) -- (p2);
\draw [name path=k1,red][extended line=0.2cm] (p1p) -- (p4);
\draw [name path=k5,red][extended line=0.2cm]  (p5p) -- (p3);
\draw [name path=k2][extended line=0.2cm] (p2p) -- (p5);
\draw [name path=k3,blue][extended line=0.2cm] (p3p) -- (p1);
\path [name intersections={of=k3 and k1,by=A}];
\path [name intersections={of=k2 and k4,by=B}];
\path [name intersections={of=k3 and k5,by=C}];
\path [name intersections={of=k4 and k1,by=D}];
\path [name intersections={of=k2 and k5,by=E}];
\path [name intersections={of=k3 and k2,by=F}];
\path [name intersections={of=k3 and k4,by=G}];
\path [name intersections={of=k4 and k5,by=H}];
\path [name intersections={of=k5 and k1,by=I}];
\path [name intersections={of=k2 and k1,by=J}];
\draw[white,fill=white] (A) circle (.6ex);
\draw[white,fill=white] (B) circle (.6ex);
\draw[white,fill=white] (C) circle (.6ex);
\draw[white,fill=white] (D) circle (.6ex);
\draw[white,fill=white] (E) circle (.6ex);
\draw (p1) [blue]-- node[below] {\tiny \color{blue}$l_3$} (F); 
\draw (p2) [blue]--node[right] {\tiny $l_1$} (G); 
\draw [red] (p3) -- node[below] {\tiny $l_9$} (H); 
\draw [red](p4) -- node[left] {\tiny $l_7$} (I); 
\draw (p5) -- node[right] {\tiny $l_{5}$} (J); 
\end{tikzpicture}
\end{center}}
\noindent The symmetry group of this configuration is $C_5$, and there are exactly two such configurations in $\bP^3$, swapped by a $\fD_5$-action on $\bP^3$.

\subsection*{Linearization of $C_6$}
Let $G=\langle\sigma_6\rangle\simeq C_6\subset\Aut(X_4)$. Fix the following eight Jacobi planes:{\footnotesize{\begin{align*}
\Pi_1&=\{y_1=y_5=0\},\\
\Pi_2&=\{y_1=y_2=0\},\\
\Pi_3&=\{(-q+1)y_1+(-2q-1)y_5=qy_1-q^2y_2-y_3-y_4-q^2y_5=0\},\\
\Pi_{16}&=\{qy_1-y_2-y_3-qy_4-y_5=qy_1-q^2y_2-y_3-y_4-q^2y_5=0\},\\
\Pi_{18}&=\{qy_1-y_2-qy_3-y_4-y_5=qy_1-q^2y_2-y_3-y_4-q^2y_5=0\},\\
\Pi_{22}&=\{q^2y_1-q^2y_2-y_3-y_4-y_5=(q+2)y_1+(-q-2)y_2=0\},\\
\Pi_{25}&=\{y_1-y_5=y_2+q^2y_3+qy_4=0\},\\
\Pi_{34}&=\{q^2y_1-y_2-qy_3-q^2y_4-q^2y_5=qy_1-qy_2-y_3-y_4-y_5=0\},
\end{align*}}}

\noindent
and let 
$D_8=\Pi_1+\Pi_2+\Pi_3+\Pi_{16}+\Pi_{18}+\Pi_{22}+\Pi_{25}+\Pi_{34}$. 
Here we keep the enumeration of Jacobi planes in $X_4$ as set in {\tt Magma}, and recorded in \cite{CTZ-tables}. 
Let
$$
\mathcal{M}_4=|4(-K_{X_4})-D_8|,
$$
it is a $G$-invariant four-dimensional mobile linear system,
since  the class $[D_8]\in\mathrm{Cl}(X_4)$ is $G$-invariant.
Choosing an appropriate basis $f_1,\ldots,f_5$ of $\cM_4$ (see \cite{CTZ-tables} for explicit equations of polynomials of the choice), we obtain an explicit rational map $X_4\dasharrow\mathbb{P}^4$ given by 
$$
[y_1:y_2:y_3:y_4:y_5]\mapsto [f_1:f_2:f_3:f_4:f_5],
$$
whose image is a quadric threefold $Q\subset\mathbb{P}^4$ with equation
\begin{align*}
&3y_1^2-3y_1y_2+3y_2^2+(2q-2)y_1y_3+(-4q-2)y_2y_3+\\
&(5q+4)y_1y_4+(-q-5)y_2y_4+(-q-2)y_3y_4+(2q+1)y_4^2+\\
&(-q+1)y_1y_5+(2q+1)y_2y_5+(q-1)y_3y_5+3y_4y_5+(-2q-1)y_5^2=0.
\end{align*}
Note that $Q$ is a quadric cone with vertex at  $[0:0:q^2:q:1]$.

We have constructed a $G$-equivariant rational map $\chi\colon X_4\dasharrow Q$,
where the induced $G$-action on $Q$ is given by the projective transformation
$$
\begin{pmatrix}
    -3-3 q  & 0 & q -1 & q +2 & -2 q -1 \\
 3 & 0 & q +2 & 4 q +2 & 4 q +5\\
 1 & q +1 & -2 q +2 & -2 q -1 & q +2 \\
 2 q +2 & -q -3 & 2 q -2 & 2 q -2 & -q +1 \\
 q +3 & 3 q -1 & 3 q +3 & 6 q +3 & 3 
\end{pmatrix}.
$$
Using {\tt Magma}, one can check that $\chi$ is birational.
Note that $G$ fixes the point 
\begin{align}\label{eq:C6fixedpoint}
  [-2+q:-3q-8:3q+1:3q+1:7]\in Q 
\end{align}
Thus, composing $\chi$ with the projection $Q\dasharrow\mathbb{P}^3$ from this point,
we obtain a $G$-equivariant birational map $X_4\dasharrow \mathbb{P}^3$, which gives a linearization of the subgroup $G\simeq C_6$.

The linearization of $G$ can also be proved as follows.
Consider the following six lines in $\mathbb{P}^4$:
\begin{align*}
L_1&=\{y_4-y_5=y_3-y_5=y_1+qy_2=0\},\\
L_2&=\{2y_3-y_4-y_5=2y_2-y_4+y_5=2y_1-q^2y_4+q^2y_5=0\},\\
L_3&=\{y_4-qy_5=2y_3+q^2y_5=y_1+q^2y_2-(q+2)y_5=0\},\\
L_4&=\{y_4+q^2y_5=y_3=y_1+q^2y_2+qy_5=0\},\\
L_5&=\{y_3+qy_4+y_5=y_2-y_4-q^2y_5=y_1+2qy_4+(q+2)y_5=0\},\\
L_6&=\{y_4-qy_5=y_3-qy_5=y_1+qy_2=0\}.
\end{align*}
They are contained in the quadric cone $Q$, and they form a hexagon.
Now, consider the following two conics in $Q$: 
\begin{align*}
R&=\{2y_3-q^2y_4-qy_5=2y_1+2qy_2-q^2y_4+qy_5=0\}\cap Q,\\
R^\prime&=\{y_3-y_4=y_1+q^2y_2+qy_4-qy_5=0\}\cap Q.
\end{align*}
Both $R$ and $R^\prime$ are smooth, they intersect transversally at the $G$-fixed point \eqref{eq:C6fixedpoint},
and they do not contain the singular point of the cone $Q$.
Note that 
\begin{itemize}
\item $R$ contains the intersection points $L_2\cap L_3$, $L_4\cap L_5$, $L_6\cap L_1$,
\item $R^\prime$ contains the intersection points  $L_1\cap L_2$, $L_3\cap L_4$, $L_5\cap L_6$.
\end{itemize}
Let $Z$ be the curve $L_1+\cdots+L_6+R+R^\prime$.
Then $Z$ is a (singular) $G$-invariant curve of degree $10$ and arithmetic genus $8$, which can be visualized as follows:
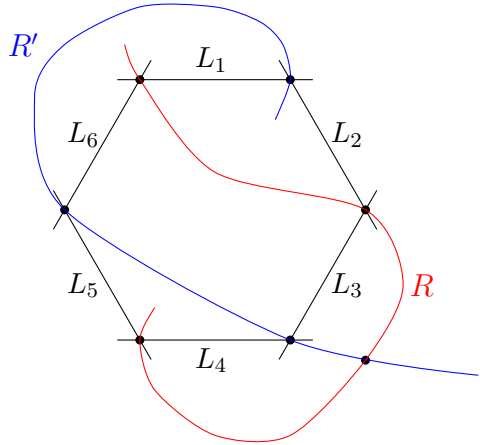
{\begin{center}
\begin{tikzpicture}
   \newdimen\R
   \R=2cm
   \draw (0:\R) \foreach \x in {60,120,...,360} {  -- (\x:\R) };
   \draw (-1.3,1.732) -- (-1,1.732) ;
   \draw (1,1.732) -- (1.3,1.732) ;
   \draw (1,-1.732) -- (1.3,-1.732) ;
   \draw (-1,-1.732) -- (-1.3,-1.732) ;
   \draw (-2,0) -- (-2.15,-0.259) ;
   \draw (-2,0) -- (-2.15,0.259) ;
   \draw (-1,1.732) -- (-0.85,1.991) ;
   \draw (-1,-1.732) -- (-0.85,-1.991) ;
   \draw (1,1.732) -- (0.85,1.991) ;
   \draw (1,-1.732) -- (0.85,-1.991) ;
    \draw (2,0) -- (2.15,-0.259) ;
   \draw (2,0) -- (2.15,0.259) ;

    \node at (-2,0)[circle,fill=black, inner sep=1.2pt, label=left:]{};
    \node at (-0.05,1.55)[label=above:{\small $L_1$}]{};
    \node at (-1,1.732)[circle,fill=black, inner sep=1.2pt, label=left:]{};
    \node at (1.25,1.0)[label=right:{\small $L_2$}]{};
    \node at (1.25,-1.0)[label=right:{\small $L_3$}]{};
    \node at (-0.05,-1.55)[label=below:{\small $L_4$}]{};
    \node at (-1.25,-1.0)[label=left:{\small $L_5$}]{};
    \node at (-1.25,1.0)[label=left:{\small $L_6$}]{};

    \node at (1,1.732)[circle,fill=black, inner sep=1.2pt, label=right:]{};
    \node at (2,0)[circle,fill=black, inner sep=1.2pt, label=right:]{};
    \node at (1,-1.732)[circle,fill=black, inner sep=1.2pt, label=below:]{};
     \node at (-2.05,2.2)[label={[blue]left:${\small R'}$}]{};
     \node at (2.3,-1)[ label={[red]right:${\small R}$}]{};

    \node at (-1,-1.732)[circle,fill=black, inner sep=1.2pt, label=left:]{};
    \node at (2,-2)[circle,fill=black, inner sep=1.2pt, label=below:]{};
    
\draw[blue] plot[smooth] coordinates 
{(0.8,1.2) (1,1.732) (0.9,2.2)(0.6,2.5)(0,2.7) (-1,2.7) (-2,2.2)(-2.4,1.5)(-2,0) (1,-1.732) (3.5,-2.2)};
\draw[red] plot[smooth] coordinates 
{(-1.2,2.2)(-1,1.732)(0,0.5)(2,0)(2.5,-1)(2,-2)(1,-3)(0,-3)(-0.8,-2.4)(-1,-1.732)(-0.8,-1.3)};
\end{tikzpicture}
\end{center}}
Let $\mathcal{M}_3$ be the linear subsystem in $|\mathcal{O}_{Q}(3)|$ consisting of all surfaces that contain $Z$.
Then $\mathcal{M}_3$ is a $G$-invariant mobile four-dimensional linear system, 
and it gives a $G$-equivariant rational map $Q\dasharrow \mathbb{P}^4$ 
whose image is an irreducible quartic threefold $X_4^\prime$ that has $45$ nodes.
This gives $X_4^\prime\simeq X_4$ \cite{dejong},
so choosing a suitable basis for $\mathcal{M}_3$ (see \cite{CTZ-tables} for equations of the choice), 
we get $X_4'=X_4$. 

This gives us a $G$-equivariant rational map $\rho\colon Q\dasharrow X_4$.
One can check that $\rho$ is birational.
Moreover, choosing a suitable basis of $\mathcal{M}_3$, we get $\rho=\chi^{-1}$.
We conclude that $G$ is linearizable;
$\mathsf{PSp}_4(\bF_3)$ contains four subgroups isomorphic to $C_6$ (up to conjugation),
and we already proved that three of them are not linearizable.

Note that the indeterminacy of $\rho$ can be resolved via the following $G$-equivariant commutative diagram:
\begin{equation}
\label{eq:C6}
\xymatrix{
U\ar@{->}[d]_{\eta}\ar@{->}[rr]^{\pi}&&V\ar@{->}[d]^{\phi}\\%
X_4&&Q\ar@{-->}[ll]^{\rho}}
\end{equation}
where $\phi$ is a blow up of the singular curve $R+R^\prime$, 
$\pi$ is a blow up of a nodal curve of arithmetic genus $1$ that is a union of the proper transforms of the lines $L_1,\ldots,L_6$
and the fibers of the morphism $\phi$ over the points 
$L_1\cap L_2$, $L_2\cap L_3$, $L_3\cap L_4$, $L_4\cap L_5$, $L_5\cap L_6$, $L_6\cap L_1$,
and $\eta$ is a birational morphism
that contracts $31$ disjoint  curves to $31$ nodes of the quartic $X_4$.
The threefold $V$ has two nodes, and $U$ has $14$ nodes, since the curves blown up by $\pi$ form a dodecagon.

\subsection*{Linearization of $C_2^2$}
In the previous subsection, we presented an explicit $C_6$-equivariant birational map $\chi\colon X_4\dasharrow Q$,
where $Q$ is a quadric cone in $\mathbb{P}^4$.
Since $C_6$ fixes a point in $Q$, this gave us a linearization of~$C_6$.
Let us use $\chi$ to construct a $\fD_6$-birational map from $X_4$ to a smooth quadric threefold in $\mathbb{P}^4$,
which will give us a  linearization of  $C_2^2\subset\fD_6$.
Set $G=\langle \sigma_6,\sigma_2\rangle\simeq\mathfrak{D}_6$. Recall that our birational map $\chi$ is $\langle\sigma_6\rangle$-equivariant,
but not $G$-equivariant,
since the involution $\sigma_2$ acts birationally on the quadric cone $Q$ via
$$
[y_1:y_2:y_3:y_4:y_5]\mapsto [t_1:t_2:t_3:t_4:t_5],
$$
where 
\begin{multline*}
t_1=19y_1y_3+(18q+7)y_1y_4+(-6q+4)y_1y_5+(-5q-3)y_2y_3+\\
(-5q-3)y_2y_4+(4q-9)y_2y_5+(2q-14)y_3^2+(2q+5)y_3y_4+\\
(11q+18)y_3y_5+(3q-2)y_4^2+(-9q+6)y_4y_5+(-9q-13)y_5^2,
\end{multline*}
\begin{multline*}
t_2=(3q+17)y_1y_3+(6q-4)y_1y_4+(-3q+2)y_1y_5+(-q-12)y_2y_3+\\
(-10q-6)y_2y_4+(-q-12)y_2y_5+(4q-28)y_3^2+(-5q+16)y_3y_4+\\
19qy_3y_5+(12q+11)y_4^2+(12-18q)y_4y_5+(-12q-11)y_5^2,
\end{multline*}
\begin{multline*}
t_3=(12q+11)y_1y_3+(-2q-5)y_1y_4+(5q+3)y_1y_5+\\
(-q-12)y_2y_3+(-3q+2)y_2y_4+(-2q-5)y_2y_5+(8q+1)y_3^2+\\
(-2q-5)y_3y_4+(2q+5)y_3y_5+(-3q+2)y_4^2+(-5q-3)y_5^2,
\end{multline*}
\begin{multline*}
t_4=(20q+12)y_1y_3+(-9q-13)y_1y_4+(4q+10)y_1y_5+\\
(-2q-5)y_2y_3+(5q+3)y_2y_4+(-9q-13)y_2y_5+(8q+1)y_3^2+(-9q-13)y_3y_4+\\
(12q-8)y_3y_5+(-4q+9)y_4^2+(6q+15)y_4y_5+(-13q-4)y_5^2,
\end{multline*}
\begin{multline*}
t_5=(q+1)y_1y_3+(-q-12)y_1y_4+(-3q+2)y_1y_5+(-9q-13)y_2y_3+\\
(4q+10)y_2y_4+(-q-12)y_2y_5+(-10q+13)y_3^2+(23q-9)y_3y_4+\\
(10q+6)y_3y_5+(-17q-14)y_4^2+(6q+15)y_4y_5+(-12q-11)y_5^2.
\end{multline*}
However, the linear system $\mathcal{M}_3\subset|\mathcal{O}_{Q}(3)|$ constructed in the previous subsection is $G$-invariant,
and  $G$ acts biregularly on the threefold $V$.
Thus, the birational map $\pi\circ\eta^{-1}\colon X_4\dasharrow V$ in the diagram \eqref{eq:C6} is also $G$-equivariant.
Observe that $\mathrm{rk}\,\mathrm{Cl}(V)^{G}=\mathrm{rk}\,\mathrm{Pic}(V)^{G}=1$, 
so $V$ is a $G$-Mori fiber space, cf. \cite[Theorem~6.5(iii)]{Prokhorov2013}.

Now, let $\mathrm{pr}\colon Q\dasharrow\mathbb{P}^2$ be the projection of the quadric cone $Q$ from the line 
passing through the vertex of $Q$ and the $G$-fixed point \eqref{eq:C6fixedpoint}. 
Choosing appropriate coordinates on $\mathbb{P}^2$, the projection map $\mathrm{pr}$ is given by
$$
\begin{pmatrix}
    3&0&0\\
    0&3&0\\
    0&0&3\\
    q+2&5q+4&-6q\\
    -q+1&q+5&3q^2
\end{pmatrix}.
$$
One can check that 
$$
\mathrm{pr}\times(\mathrm{pr}\circ \iota)\colon Q\dasharrow \bP^2_{x_0,x_1,x_2}\times\bP^2_{z_0,z_1,z_2}
$$
gives a birational map $\varrho\colon Q\to W$, where $W\subset \bP^2\times\bP^2$ is a smooth divisor of bidegree $(1,1)$ 
that is given by
\begin{align*}
&6x_0z_0-3x_0z_1+2(q-1)x_0z_2-3x_1z_0-3x_1z_1+(5q+4)x_1z_2\\
+&2(q-1)x_2z_0+(5q+4)x_2z_1+(-2q+5)x_2z_2=0.    
\end{align*}
The birational map $\varrho$ is $G$-equivariant,
and \mbox{$\mathrm{rk}\,\mathrm{Cl}(W)^{G}=1$}, so $W$ is a $G$-Mori fiber space.
This gives the $G$-equivariant commutative diagram:
$$
\xymatrix{
&\widetilde{V}\ar@{-->}[rr]\ar@{->}[dl]_{\alpha}&&\widetilde{W}\ar@{->}[dr]^{\beta}&\\%
V\ar@{-->}[rrrr]^{\varrho}&&&& W}
$$
where $\alpha$ is the blowup of both singular points of $V$,  
$\widetilde{V}\dasharrow\widetilde{W}$ is a flop in the strict transform of the line in $Q$ passing through
it vertex and the point \eqref{eq:C6fixedpoint}, and $\beta$ is a blow up of a $G$-irreducible smooth curve
consisting of four irreducible components such that one of them is 
the fiber of the projection to the second factor $W\to\mathbb{P}^2$ over $[1:3q + 2:3]$.

The group $G$ leaves invariant the curve 
$$
\{ 3x_1-(3q+2)x_2=3z_1-(3q+2)z_2=0\}\subset W,
$$
which is a curve of degree $(1,1)$. 
Blowing up this curve $\gamma\colon\widehat{W}\to W$,
we obtain the following (classical) $G$-Sarkisov link:
$$
\xymatrix{
&\widehat{W}\ar@{-->}[rr]\ar@{->}[dl]_{\gamma}&&\widehat{X}_2\ar@{->}[dr]^{\delta}&\\%
W&&&& X_2}
$$
where $X_2$ is a smooth quadric 3-fold in $\mathbb{P}^4$, and $\delta$ is a blow up of two disjoint lines in $X_2$. 

To describe the map $W\dasharrow X_2$ explicitly, observe that $G$
leaves invariant the affine chart of $W$ given by 
$$3x_1-(3q+2)x_2\ne 0\quad \text{ and } \quad 3z_1-(3q+2)z_2\ne0,
$$
which is an affine quadric 3-fold, 
whose $G$-equivariant compactification is $X_2$. 
Hence, choosing appropriate coordinates on $\mathbb{P}^4$, we may assume that $X_2$ is given by 
$$
y_1y_4-y_2y_3+y_5^2=0,
$$
and the induced $G$-action on $X_2$ is generated by 
\begin{align*}
[y_1:y_2:y_3:y_4:y_5]&\mapsto [-q^2y_1: q^2y_2: qy_3: -qy_4: -y_5],\\
[y_1:y_2:y_3:y_4:y_5]&\mapsto [y_4:y_3:y_2:y_1:y_5].
\end{align*}
In particular, we see that the subgroup $\langle\sigma_2,\sigma_6^3\rangle\simeq C_2^2$ is linearizable,
because the corresponding $C_2^2$-action on $X_2$ has a fixed point.

\subsection*{Linearization of $C_4$}
Now, we let $G\simeq C_4$ be the subgroup in $\mathsf{PSp}_4(\bF_3)$ generated by $\sigma_4$. Consider the $G$-orbit of four planes given by 
{\footnotesize\begin{align*}
    \Pi_3&=\{(-q + 1)y_1+(-2q - 1)y_5=0,
    qy_1 -q^2y_2 - y_3 - y_4 -q^2y_5=0\},\\
    \Pi_{8}&=\{q^2y_1 - qy_2 -q^2y_3 - qy_4 - qy_5=0,
    (q + 2)y_1 + (2q + 1)y_2=0\},\\
    \Pi_{12}&=\{q^2y_1 - qy_2 - qy_3 -q^2y_4 -qy_5=0,
    (q + 2)y_1 + (2q + 1)y_5=0\},\\
     \Pi_{15}&=\{q^2y_1 - qy_2 - y_3 - y_4 - qy_5=0, qy_1-q^2y_2 - qy_3 -q^2y_4-q^2y_5=0\},
\end{align*}}

\noindent
where $q$ is a primitive cube root of unity, as above.
Then projection from each of these planes produces a map $X_4\dasharrow \bP^1$. 
The product of these four projections results in a $G$-equivariant map
$$
\pi: X_4\dasharrow \bP^1_{x_0,x_1}\times\bP^1_{z_0,z_1}\times\bP^1_{u_0,u_1}\times\bP^1_{t_0,t_1}.
$$
One can check that this map is birational onto its image. Choosing appropriate coordinates, the image $V$ is a divisor in $(\bP^1)^4$ given by 
$$
x_0z_0u_0t_0 + x_0z_0u_0t_1 + x_0z_1u_0t_1 + x_1z_1u_0t_1 + x_1z_1u_1t_1=0.
$$ 
Notice that $V$ has $5$ singular points, which are ordinary double points. 
Observe also that $\mathrm{rk}\,\mathrm{Pic}(V)^{G}=1$, but $\mathrm{rk}\,\mathrm{Cl}(V)^{G}=2$, so $V$ is not a $G$-Mori fiber space. 

Let us find a $G\mathbb{Q}$-factorialization of $V$.
To do this, we choose another $G$-orbit of four planes in $X_4$:
{\footnotesize
\begin{align*}
    \Pi_1&=\{(-q + 1)y_1 + (-2q - 1)y_5=(q + 2)y_1 + (2q + 1)y_5=0\},    \\
    \Pi_{5}&=\{(-q + 1)y_1 + (-2q - 1)y_5=0, qy_1 + (q + 1)y_2 - qy_3 -q^2y_4 -q^2y_5=0\},\\
    \Pi_{14}&=\{q^2y_1 - qy_2 -q^2y_3 - qy_4 - qy_5=0, (q + 2)y_1 + (2q + 1)y_5=0\},\\ \Pi_{20}&=\{q^2y_1 - qy_2 -q^2y_3 - qy_4 - qy_5=0, qy_1 -q^2y_2 - qy_3 -q^2y_4 -q^2y_5=0\}.
\end{align*}
}
One can check they are not contracted under the map $\pi$,
and the $G$-invariant divisor $\pi(\Pi_1)+\pi(\Pi_{5})+\pi(\Pi_{14})+\pi(\Pi_{20})$ is not $\mathbb{Q}$-Cartier.
The linear system 
$$
|-K_V-\pi(\Pi_1)-\pi(\Pi_5)-\pi(\Pi_{14})-\pi(\Pi_{20})|
$$
is $G$-invariant, and its projective dimension is $5$. 
Moreover, under the choice of basis 
$$
x_0z_0u_0t_1,\,\, \,x_0z_1u_0t_0,\,\,\, x_0z_1u_0t_1, \,\,\,x_0z_1u_1t_1,\,\,\, x_1z_1u_0t_1,\,\,\, x_1z_1u_1t_1,
$$
it gives a $G$-equivariant birational map $\rho: V\dasharrow X_{2,2}$,
where $X_{2,2}$ is a complete intersection in $\bP^5$, with equations:
$$
v_4v_5 - v_3v_6=v_1v_2 + v_1v_3 + v_3^2 + v_3v_5 + v_4v_5=0.
$$
The induced $G$-action on $X_{2,2}$ is given by 
$$
\begin{pmatrix}
          0&           0 &          0 &          1&           0&           0\\
          0&           0 &          0 &          0&           1&          -1\\
          0&  -q^2&          -1 &          1&           1 &          0\\
          0&      q&           0&          0&           0 &          0\\
q^2  &         0&           0&           0&           0 &          1\\
          0&           0&           0 &          0 &          0 &          1
\end{pmatrix}.
$$
Note that this $X_{2,2}$ has $5$ singular points and is not toric, so it is different from the toric intersection of two quadrics with $6$ singular points that appeared in Section~\ref{sect:tori} and ~\ref{sect:segre}. 

The birational map $\rho$ fits the following $G$-equivariant commutative diagram:
$$
\xymatrix{
&\widetilde{V}\ar@{->}[dl]_\alpha\ar@{->}[dr]^\beta&&\widehat{V}\ar@{->}[dr]^\delta\ar@{->}[dl]_\gamma\\%
\bP^1&&V\ar@{-->}[rr]^{\rho}&& X_{2,2}}
$$
where $\beta$ and $\gamma$ are small $G$-equivariant birational morphisms
that resolve $4$ singular points of $V$ forming one $G$-orbit,
$\delta$ is a blow up of $4$ singular points of $V$ forming one $G$-orbit,
and $\alpha$ is a fibration into Del Pezzo surfaces of degree $4$.
The small birational morphisms $\beta$ and $\gamma$ are $G\mathbb{Q}$-factorializations of $V$,
and the composition $\alpha\circ\beta^{-1}\circ\gamma\circ\delta^{-1}$ is 
given by the projection $\mathbb{P}^5\dasharrow\mathbb{P}^1$ from the three-dimensional linear subspace in $\mathbb{P}^5$ that contains $4$ singular points of $X_{2,2}$ blown up by $\delta$.
A similar $G$-Sarkisov link appeared in the proof of \cite[Lemma~2.16]{Avilov2016}.

Now, we let $P=[0:0:0:0:1]$. Then $P$ is a $G$-fixed singular point of $X_{2,2}$.
Projection from $P$ gives a $G$-birational map from $X_{2,2}$ to a smooth quadric threefold $X_2\subset\bP^4$
that fits the following $G$-equivariant commutative diagram:
$$
\xymatrix{
&Y\ar@{->}[dl]\ar@{->}[dr]&\\%
X_{2,2}\ar@{-->}[rr]&&  X_{2}}
$$
where $Y\to X_{2,2}$ is the blowup of the point $P$,
and $Y\to X_{2}$ is the blow up of a singular connected curve of arithmetic genus $1$ and degree $4$,
which is a union of four lines. Note that $Y$ is a singular Fano threefold in the deformation family \textnumero 2.23,
and the constructed $G$-Sarkisov link is a degeneration of a classical Sarkisov link
that blows up  a smooth quadric threefold along a smooth quartic elliptic curve.

Since the $G$-action on a smooth quadric threefold has a fixed point, the action of $G$ is linearizable.

\section{Equivariant birational rigidity}
\label{sect:Burk-rigid}

Let $X_4\subset \mathbb{P}^4$ be the Burkhardt quartic and $G\subseteq \mathrm{Aut}(X_4)$ be such that $\mathrm{rk}\,\mathrm{Cl}^G(X_4)=1$.
In this section we prove Proposition~\ref{thm:Burkhard-BSR}, i.e., 
we show that $X_4$ is $G$-birationally super-rigid. 
We start by recalling several well-known geometric facts about $X_4$, 
and proving three technical lemmas.

The quartic $X_4$ has $45$ isolated ordinary double points (nodes).
One can also check that
\begin{itemize}
\item a line in $\mathbb{P}^4$ can contain $1$, $2$ or $3$ nodes,

\item a plane in $\mathbb{P}^4$ can contain $1$, $2$, $3$, $4$, $6$ or $9$ nodes,

\item a hyperplane in $\mathbb{P}^4$ can contain $1$, $2$, $3$, $4$, $7$, $10$, $12$ or $18$ nodes.
\end{itemize}
Planes in $\mathbb{P}^4$ containing $9$ nodes of $X_4$ are called Jacobi planes --- these planes are contained in $X_4$.
The threefold $X_4$ contains $40$ Jacobi planes,
each of these $40$ planes contains exactly $9$ nodes of $X_4$,
and there are exactly $8$ planes in $X_4$ that pass through a given node.
The union of all planes in $X_4$ is a divisor in $|10(-K_{X_4})|$, which we denote by~$\mathbf{J}$.
Similarly, hyperplanes in $\mathbb{P}^4$ containing $18$ nodes of $X_4$ are called Steiner hyperplanes ---
their intersections with $X_4$ split into unions of $4$ Jacobi planes.
We will call such unions of $4$ Jacobi planes \emph{tetrahedra}.
There are $40$ Steiner hyperplanes, so $X_4$ contains $40$ tetrahedra.

\begin{lemm}
\label{lemma:tetrahedra}
Let $\Sigma$ be a subset of the singular locus $\mathrm{Sing}(X_4)$, of cardinality  $s=|\Sigma|\ge 1$.
Suppose that at least one of the following  conditions is satisfied:
\begin{enumerate}
\item[$(\mathrm{1})$] $s\leqslant 4$,
\item[$(\mathrm{2})$] $s\in\{5,6\}$, the set $\Sigma$ is contained in a plane in $\mathbb{P}^4$, no $3$ points in $\Sigma$ are collinear,
\item[$(\mathrm{3})$] $s=7$, the set $\Sigma$ is contained in a hyperplane in $\mathbb{P}^4$, no $4$ points of the set $\Sigma$ are contained in a plane in $\mathbb{P}^4$.
\end{enumerate}
Then $X_4$ contains a tetrahedron that is disjoint from $\Sigma$.
\end{lemm}

\begin{proof}
Computer computations.
\end{proof}

\begin{lemm}
\label{lemma:lines-conics-twisted-cubic}
Let $C$ be an irreducible curve in $X_4$ of degree $d\leqslant 3$, 
let $\Sigma=C\cap\mathrm{Sing}(X_4)$, and let $s=|\Sigma|$.
If $C$ is a twisted cubic curve, we also suppose that $s\not\in\{5,6\}$.
Then there is a Jacobi plane $\Lambda\subset X_4$ such that $\Lambda\cap C$ contains a smooth point of $X_4$.
\end{lemm}

\begin{proof}
Since the locus $\mathrm{Sing}(X_4)$ is an intersection of cubics in $\mathbb{P}^4$, we see that $s\leqslant 3d$.
Thus, if $d\leqslant 2$, then $X_4$ contains a tetrahedron $T$ that is disjoint from $\Sigma$ by Lemma~\ref{lemma:tetrahedra},
so that $T\cap C$ contains a smooth point of $X_4$, since $T\cap C\ne\varnothing$.
This proves the lemma in the case $d\leqslant 2$.
Hence, we may assume that $C$ is  either a plane cubic or a twisted cubic.

Suppose that $C$ is a plane cubic. Let $\Pi$ be the plane in $\mathbb{P}^4$ that contains $C$.
If $\Pi\subset X_4$, we are done. If $\Pi\not\subset X_4$,
then $X_4\vert_{\Pi}=C+\ell$ for some line $\ell$,
which gives $s\leqslant 4$, since $\Sigma\subset\mathrm{Sing}(C+\ell)$.
Now, the required assertion follows from Lemma~\ref{lemma:tetrahedra}.

To complete the proof, we may assume that $C$ is a twisted cubic curve.
If $s\leqslant 4$ or $s=7$, the assertion follows from Lemma~\ref{lemma:tetrahedra}.
Thus, we may assume that $s\in\{8,9\}$.
Let $f\colon\widetilde{X}_4\to X_4$ be the blow up of all singular points of $X_4$,
let $\widetilde{\mathbf{J}}$ be the strict transform on $\widetilde{X}_4$ of the divisor~$\mathbf{J}$,
let $E$ be the union of all $f$-exceptional prime divisors,
and let $\widetilde{C}$ be the strict transform on $\widetilde{X}_4$ of the curve $C$. Then
$$
0\leqslant\widetilde{\mathbf{J}}\cdot\widetilde{C}=\Big(f^*\big(-10K_{X_4}\big)-4E\Big)\cdot\widetilde{C}=30-4E\cdot\widetilde{C}\leqslant 30-4s\leqslant -2,
$$
which is absurd. This completes the proof of the lemma.
\end{proof}

\begin{lemm}
\label{lemma:twisted-cubics}
Let $C\subset X_4$ be a twisted cubic curve and $\Sigma=C\cap\mathrm{Sing}(X_4)$. Set $s=|\Sigma|$
and let $\mathcal{M}$ be a non-empty mobile linear subsystem in $|-nK_{X_4}|$ for some positive integer $n$.
Suppose that $s\leqslant 6$. Then $\mathrm{mult}_C(\mathcal{M})\leqslant n$.
\end{lemm}

\begin{proof}
If $s\geqslant 1$, let $g\colon\overline{X}_4\to X_4$ be the blow up of $\Sigma$.
If $s=0$, we let $\overline{X}_4=X_4$ and $g=\mathrm{Id}_{X_4}$.
Let $f\colon\widetilde{X}_4\to\overline{X}_4$ be the blow up of the strict transform on $\overline{X}_4$ of the twisted cubic curve $C$,
let $F$ be the $f$-exceptional surface, let~$E_1,\ldots,E_s$ be the $(f\circ g)$-exceptional prime divisors that are mapped to the~subset $\Sigma$,
and let $\widetilde{\mathcal{M}}$ be the strict transform on the~threefold $\widetilde{X}_4$.
Set  $m=\mathrm{mult}_C(\mathcal{M})$. Then
$$
\widetilde{\mathcal{M}}\sim_{\mathbb{Q}}f^*\big(-nK_{X_4}\big)-mF-\sum_{i=1}^{s}a_iE,
$$
where $a_1,\ldots,a_s\in \bZ_{\ge 0}$.
We have to show that $m\leqslant n$.

If $s\geqslant 1$, then each $E_i$ is a Del Pezzo surface of degree~$7$ and
$$
\widetilde{\mathcal{M}}\big\vert_{E_i}\sim -mF\big\vert_{E_i}-nE_i\big\vert_{E_i},
$$
which implies that $m\leqslant 2a_i$. So, if $s\geqslant 1$, then
$a_i\geqslant \frac{m}{2}$ for every $i$.

Set $H=(f\circ g)^*(-K_{X_4})$. Then $|2H-F-\sum_{i=1}^sE_i|$ does not have base curves,
because the~curve $C$ is cut out by quadrics in $\mathbb{P}^4$.
Let $D$ be a general surface in $|2H-F-\sum_{i=1}^sE_i|$.
Then $D$ is nef. Let $\widetilde{M}_1$ and  $\widetilde{M}_2$ be general surfaces in $\widetilde{\mathcal{M}}$.
Then $D\cdot\widetilde{M}_1\cdot\widetilde{M}_2\geqslant 0$.

Let us compute $D\cdot\widetilde{M}_1\cdot\widetilde{M}_2$. We have
\begin{align*}
H^3&=4,       & E_i\cdot H^2&=0, & F\cdot E_i^2&=0, & F^3&=s-1, & E_i\cdot F^2&=-s,\\
H\cdot F^2&=-3, & F\cdot H^2&=0, & E_i\cdot F\cdot H&=0, & E_i^3&=2s, & H\cdot E_i^2&=0.
\end{align*}
Then
$$
0\leqslant D\cdot\widetilde{M}_1\cdot\widetilde{M}_2=8-5m^2+2\Big(\sum_{i=1}^sa_i-3\Big)m-2\sum_{i=1}^{s}a_i^2.
$$
This gives $m\leqslant n$, since $s\leqslant 6$ and $a_i\geqslant \frac{m}{2}$ for $i\in\{1,\ldots,s\}$.
\end{proof}

Now, we are ready to prove that $X_4$ is $G$-birationally super-rigid. 
Suppose it is not.
By Corollary~\ref{coro:NFI}, there is a non-empty $G$-invariant mobile linear subsystem
$\mathcal{M}\subset |-nK_{X_4}|$, for some positive integer $n$, such that
the singularities of the pair $(X_4,\frac{1}{n}\mathcal{M})$ are not canonical.
Let us seek a contradiction.

Set $\lambda=\frac{1}{n}$. Let $Z$ be a center of non-canonical singularities of the log pair $(X_4,\lambda\mathcal{M})$,
let $M_1$ and $M_2$ be two general surfaces in the linear system $\mathcal{M}$.
If $Z$ is a smooth point of $X_4$,
then it follows from \cite{Pukhlikov} or \cite[Corollary 3.4]{Co00} that
$$
\Big(M_1\cdot M_2\Big)_Z>\frac{4}{\lambda^2}=4n^2,
$$
which leads to a contradiction:
$$
4n^2=\frac{4}{\lambda^2}=H\cdot M_1\cdot M_2\geqslant\Big(M_1\cdot M_2\Big)_Z>\frac{4}{\lambda^2}=4n^2,
$$
where $H$ is a general hyperplane section of $X_4$ passing through $P$.
Thus, either $Z$ is a singular point of $X_4$, or $Z$ is an irreducible curve. 

Suppose that $Z$ a singular point of $X_4$.
Let $f\colon\widetilde{X}_4\to X_4$ be the blow up of this point,
let $E$ be the $f$-exceptional surface, let $\widetilde{\mathcal{M}}$ be the strict transform on $\widetilde{X}_4$
of the~linear system $\mathcal{M}$, and let $\widetilde{M}$ be a general surface in  $\widetilde{\mathcal{M}}$.
Then
$$
\widetilde{\mathcal{M}}\sim_{\mathbb{Q}}f^*\big(-nK_{X_4}\big)-aE,
$$
for some  integer $a>n$, by \cite[Theorem~1.7.20]{CheltsovUMN} or \cite[Theorem~3.10]{Co00}.
Now, let $\Pi$ be a Jacobi plane in $X_4$ that contains $Z$,
let $L$ be a general line in $\Pi$ that contains $Z$,
and let $\widetilde{L}$ be its strict transform on~$\widetilde{X}_4$.
Then $\widetilde{L}\not\subset\widetilde{M}$, so that $0\leqslant\widetilde{M}\cdot\widetilde{L}=n-a<0$,
which is absurd. 

Thus, $Z$ is an irreducible curve. Then $\mathrm{mult}_{Z}(\mathcal{M})>\frac{1}{\lambda}=n$.
Write
$$
M_1\cdot M_2=mZ+\Delta,
$$
where $m$ is a positive integer such that  $m>n^2$, and $\Delta$ is an effective one-cycle whose support does not contain $Z$.
Then
$$
4n^2=\!\frac{4}{\lambda^2}\!=\!-K_{X_4}\cdot M_1\cdot M_2\!=\!m\mathrm{deg}(Z)-K_{X_4}\cdot\Delta\geqslant m\mathrm{deg}(Z)>n^2\mathrm{deg}(Z),
$$
which gives $\mathrm{deg}(Z)\in\{1,2,3\}$.
As in Lemma~\ref{lemma:lines-conics-twisted-cubic}, let $\Sigma=Z\cap\mathrm{Sing}(X_4)$, and set $s=|\Sigma|$.
If $Z$ is a twisted cubic, then $s\not\in\{5,6\}$ by Lemma~\ref{lemma:twisted-cubics}.
Thus, it follows from Lemma~\ref{lemma:lines-conics-twisted-cubic} that $X_4$
contains a Jacobi plane $\Pi$ such that  $\Pi\cap Z$ contains a smooth point $P$ of $X_4$.
Let $\ell$ be a general line in this plane that contains $P$. Then $\ell\not\subset M_1$, so  
$$
n=\frac{1}{\lambda}=M_1\cdot\ell\geqslant\mathrm{mult}_P(M_1)\geqslant\mathrm{mult}_Z(M_1)=\mathrm{mult}_Z(\mathcal{M})>\frac{1}{\lambda}=n,
$$
which is absurd. This completes the proof of Proposition~\ref{thm:Burkhard-BSR}.

\bibliographystyle{plain}
\bibliography{segre}
\end{document}